\documentclass[11pt,reqno,a4paper]{amsart}

\usepackage{xcolor}

\usepackage{amsmath,amsfonts,amsthm,amssymb}
\usepackage{enumitem}
\usepackage[
colorlinks,
pdfpagelabels,
pdfstartview = FitH,
bookmarksopen = true,
bookmarksnumbered = true,
linkcolor = blue,
plainpages = false,
hypertexnames = false,
citecolor = red] {hyperref}

\usepackage[capitalize]{cleveref}
\usepackage[T1]{fontenc}

\usepackage{graphicx}
\usepackage{cite}

\usepackage{algorithm}

\usepackage{tabularx}  

\usepackage{mathabx}
\usepackage{esint}

\usepackage{subfigure}

\makeatletter
\def\section{\@startsection{section}{1}%
\z@{.7\linespacing\@plus\linespacing}{.5\linespacing}%
{\bfseries
\centering
}}
\def\@secnumfont{\bfseries}
\makeatother

\usepackage{graphicx}

\usepackage[capitalize]{cleveref}
\crefname{equation}{}{}

\newcounter{hipotcounter}
\renewcommand{\thehipotcounter}{(H\arabic{hipotcounter})}
\newenvironment{assump}[1]{%
  \refstepcounter{hipotcounter}
  \label{#1}
  \textbf{\thehipotcounter}\ 
  }{}

\usepackage{comment}

\includecomment{comment}
\usepackage{algorithm}
\usepackage{algorithmic}

\specialcomment{supplementary}
{\colorlet{savedcolor}{.} \color{blue} \begingroup \ttfamily      \noindent \underline{Supplementary details:} \newline \newline \footnotesize }{\endgroup   \color{savedcolor}}

\newcommand{\N}{\mathbb{N}}

\newcommand{\E}{\mathbb{E}}

\newcommand{\HeM}{\hat{H}_{M}}

\newtheorem{theorem}{Theorem}[section]
\newtheorem{lemma}[theorem]{Lemma}
\newtheorem{proposition}[theorem]{Proposition}
\newtheorem{corollary}[theorem]{Corollary}

\theoremstyle{definition}

\theoremstyle{remark}
\newtheorem{remark}[theorem]{Remark}

\numberwithin{equation}{section}
\usepackage{enumerate}

\usepackage{bbm}

\newcommand{\R}{\mathbb{R}}

\renewcommand{\P}{\mathbb{P}}

\newcommand{\norm}[1]{\ensuremath{\Vert #1 \Vert}} 

\begin{document}

\title[]{Error analysis for learning fractional stochastic differential equations with applications in neural approximations}

\author[Dehshiri]{Mahdi Dehshiri}
\address{Aalto University School of Business, Department of Information and Service Management, PO Box 11110, 00076 Aalto, Finland}

\email{mahdi.dehshiri@aalto.fi}

\author[Martinez]{Kerlyns Martinez}
\address{University of Concepci\'on, Department of Mathematical Engineering, Concepci\'on, Chile}
\email{kermartinez@udec.cl}

\author[Viitasaari]{Lauri Viitasaari}
\address{Aalto University School of Business, Department of Information and Service Management, PO Box 11110, 00076 Aalto, Finland}
\email{lauri.viitasaari@aalto.fi}

\thanks{}
\date{\today}
\begin{abstract}
This paper develops a framework for the error analysis in nonparametric model fitting of fractional stochastic differential equations based on discrete observations. We identify and quantify the main error sources —time discretization, coefficient approximation, and model fitting error —within a unified framework. Through Sobolev-type norms, we derive convergence rates that incorporate the regularity of trajectories, thereby capturing the interaction of these error components. To demonstrate the applicability of the theory, we introduce a training scheme for coefficient function estimation based on shallow neural networks and a recurrent architecture. Numerical experiments validate the theoretical findings and illustrate the effectiveness of the approach.
\end{abstract}

\maketitle

\section{Introduction}

In many fields such as finance, biology, and physics, identifying the inherent stochasticity within systems is crucial, in particular when these systems are analyzed from an open system perspective, which acknowledges the influence of external factors that can introduce randomness. Stochasticity may arise from elements extrinsic to the system or arise from uncertainties related to specific model parameters. For instance, in financial markets, the asset prices fluctuate unpredictably due to numerous variables~\cite{black1973pricing}. Similarly, in epidemiological dynamics, the spread of infections might be influenced by social behavior and random environmental factors~\cite{vstvepan2007kermack}. In fluid dynamics, turbulence represents another classic example, where chaotic flow patterns are observed~\cite{pope1994relationship}. In order to maintain the replicability of these phenomena, at least at certain scales, it is vital to understand and justify the sources of stochasticity involved. This implies a careful selection of the type of noise incorporated into the models. The typical choices range from Gaussian noise characterized by independent increments to autocorrelated Gaussian noise, jump processes, or heavy-tailed distributions that can capture extreme events and anomalies. Each choice has significant implications for the behavior of the modeled system and the accuracy of its predictions.

To capture memory and long-range dependence in complex systems across diverse fields such as finance, physics, and biology, it is nowadays common to consider fractional Brownian motion (fBm), introduced by Mandelbrot and Van Ness~\cite{mandelbrot1968fractional}. The fBm is a centered Gaussian process with stationary increments and is characterized by the Hurst parameter $H\in(0,1)$. Its covariance function is given by
\[
\mathbb{E}[B^H_t B^H_s] = \frac{1}{2}\left(t^{2H} + s^{2H} - |t-s|^{2H}\right), \quad\mbox{ for any } s,t \ge 0,
\]
which explicitly encodes long-range dependence when $H>1/2$. Applications include modeling rough volatility in finance~\cite{bayer2016pricing,bondi2024rough}, intermittency in turbulent flows~\cite{bossy2025weak}, and phenomena in materials science~\cite{avelin2025renormalized}.

In this article, we are interested in phenomena that can be represented as the solution of a stochastic differential equation (SDE) of the type
\begin{equation}\label{eq:true_model-general}
X_t^H = x_0 + \int_0^t b(s,X^H_s)ds +\int_0^t \sigma(s,X^H_s)dG_s, \quad t\in[0,T],
\end{equation}
where $b$ and $\sigma$ are suitable but unknown coefficient functions and $G$ is a random noise that is H\"older continuous of order larger than $\frac12$. In particular, we consider the noise $G$ to be the fractional Brownian motion $B^H$ with Hurst index $H>\frac12$, although our main results remain valid for more general H\"older continuous drivers $G$, see discussion after Theorem \ref{thm:main}. We stress that such models can capture rather arbitrary distributional and statistical properties, cf. \cite{shevchenko}, and hence forms an interesting class for modeling purposes.

Under mild assumptions on the coefficients, the model is well-posed and the solution is H\"older continuous~\cite{Nualart-Rascanu2002}. At this point, the exact parametrization of the model is not always known; often, the coefficients are only partially known, and the problem becomes one of estimating coefficients~\cite{cont2004nonparametric,hu2010parameter}. In many other cases, however, the coefficient parameterization is totally unknown, although heteroscedasticity, trading effects, and other structural properties may be presumed. In such scenarios, data-driven modeling approaches provide a natural framework to approximate the unknown coefficients directly from observations of the system, leveraging statistical or machine learning techniques to construct models that are consistent with observed behavior~\cite{kutz2017deep}. Recent advancements include Neural ODEs~\cite{chen2018neural} and Neural SDEs~\cite{kidger2021neural}, which complement traditional frameworks by using neural networks to parameterize the drift and diffusion coefficients. These models are trained on observed data, allowing to represent potentially nonlinear dynamics that are challenging to determine analytically. For related literature, we also mention \cite{hayashi2022fractional}, introducing ``fractional SDE-nets'' as generative models for time series with long-term memory. While their focus is on data generation rather than error analysis, it illustrates a growing interest in combining fractional dynamics with machine learning techniques.  

\noindent {\bf Contributions and related works.}  In this paper, we present a detailed error analysis related to model fitting of \eqref{eq:true_model-general} based on observations $(\hat{X}_{t_m}; 1\le m\le M)$, with $\Pi=\{t_m\}$ a partition of the time observed interval $[0,T]$ with frequency $\Delta t$. More precisely, we provide upper bounds related to three different sources of errors, see Theorem \ref{thm:main} below. 
Our main contribution is bounding the error arising from estimation of the coefficient functions $b$ and $\sigma$. Our upper bound simultaneously highlights the correct norms in which $b$ and $\sigma$ should be estimated. In particular, we prove that it suffices to consider certain weak $L^p$-type norms instead of pointwise convergence, allowing direct usage of certain neural network architectures for the estimation. On a related data-driven techniques based on machine learning, we refer to~\cite{yang2023neural,hayashi2022fractional}. Another source of error is the noise fitting error, as the noise process is unobservable. In the case of $G=B^H$ which is our main focus, this means that we need to include in the analysis the estimation of the Hurst parameter~\cite{coeurjolly2001estimating,Kubilius-Mishura-Ralchenko2017}. Finally, the third source of error is related to the time discretisation schemes, a topic that is already relatively well understood in the literature, see e.g.~\cite{Mishura-Shevchenko2008,Deya-Neuenkirch-Tindel-2012,Hu-Liu-Nualart} and references therein. 

Analysis of these three sources of error requires a careful treatment of the regularity of trajectories and the choice of the appropriate functional framework. We employ Sobolev-type norms to capture this regularity, and, along with some harmonic analysis, potential theory, and probabilistic arguments, we derive explicit convergence rates for each type of error. The results hold uniformly over a broad class of approximation procedures. That is, for any sufficiently reasonable method that approximates $(b,\sigma)$ from direct observations, we construct an approximate model $X^n$ and quantify the error in Sobolev norms. As a concrete illustration, we apply our general theory to shallow neural networks with uniformly bounded activations. 

Furthermore, we propose a training algorithm inspired by Yang et al.~\cite{yang2023neural}, which in our case explicitly incorporates the regularity of the solution and the geometry of the underlying functional space. Unlike previous works, which focus mainly on processes driven by standard Brownian motion or stable L\'evy noise with entropy-based loss functions, we consider processes with autocorrelated Gaussian noise and an adequate loss function. To formalize this, suppose we observe a data set  $\{\hat{X}_{0}, \hat{X}_{t_1}, \ldots, \hat{X}_{t_M}\},$ where the sampling error (or fitting error) is associated with the frequency of observations, and assume for simplicity that the observations are uniformly spaced, i.e. $t_m - t_{m-1} = \hat{\Delta}$ for some $\hat\Delta>0$. On each subinterval $[t_{m-1},t_m]$, we introduce a finer partition $\pi_m^N = \{t_0^{N,m} = t_{m-1}, t_1^{N,m},\ldots, t_N^{N,m} = t_m\}$ with mesh size $\|\pi_m^N\|=\Delta t$, and approximate the estimated model by its discretisation over $\pi_m^N$. The coefficients $(b_n,\sigma_n)$ are then trained by minimizing a loss function defined on the coarse grid of step size $\hat\Delta$, while the discretized dynamics converge to the continuous model as $\Delta t \to 0$. Within this framework, our analysis quantifies in a unified way the errors arising from approximation, discretisation, and noise parameter estimation.

The rest of this paper is organized as follows. Section \ref{sec:main_results} presents our main results and underlying assumptions, accompanied with discussions. We illustrate the applicability of our approach in Section \ref{sec:ML_application} where we discuss quantitative approximation using a recurrent neural network approach. Section \ref{sec:numerics} includes numerical experiments based on neural networks, showing good performance also in practice. All proofs are postponed to Appendix \ref{sec:apprendix}, which covers the time error~\ref{subsubsec:time}, the fitting error~\ref{subsubsec:fitting}, and the coefficient approximation error~\ref{subsubsec:approximation}.

\section{Main results}\label{sec:main_results}
We assume that the true observations are solution trajectories $X$ for the stochastic differential equation
\begin{equation}\label{eq:true_model}
X_t^H = x_0 + \int_0^t b(s,X^H_s)ds +\int_0^t \sigma(s,X^H_s)dB^H_s, \quad t\in[0,T],
\end{equation}
where $B^H$ is a fractional Brownian motion with Hurst index $H\in\left(\frac12,1\right)$, and $b$ and $\sigma$ are suitable but unknown coefficient functions. 

In practice, data arrive at discrete frequencies rather than as a continuous stream of information. Hence, one only observes $(\hat{X}_{t_m}; 1\le m\le M)$, with $\Pi=\{t_m\}$ as a partition of the time window $[0,T]$. Even though typically one might also have observational error, throughout we ignore it and assume we observe the true solution. That is, we have $\hat{X}_{t_m} = X_{t_m}^H$. 

If one tries to calibrate the model by using observations, on top of estimating unknown coefficient functions $b$ and $\sigma$ one also has to fit a suitable noise structure by estimating the Hurst parameter $H$. This leads to the fitting error arising from (formally) considering a process
$(X_t^{\hat{H}_M};t\in[0,T])$ as a solution to 
\begin{equation}\label{eq:statistical_mode}
X_t^{\hat{H}_M} = x_0 + \int_0^t b(s,X^{\hat{H}_M}_s)ds +\int_0^t \sigma(s,X^{\hat{H}_M}_s)dB^{\hat{H}_M}_s.
\end{equation}
Note that here we assume that the underlying source of randomness is the same for both $B^H$ and $B^{\hat{H}_M}$, simply the parameter $H$ is estimated. In the fractional Brownian motion case, this means that if one uses a kernel representation $B^H = \int K_H(t,s)dW_s$ with $W$ a standard Brownian motion, the estimated path is given by $B^{\hat{H}_M} = \int K_{\hat{H}_M}(t,s)dW_s$. In our proofs, we consider the Mandelbrot-van-Ness representation of the fBm, see \cite{mandelbrot1968fractional}.

The second error, called the approximation error in the sequel, arises from the estimation of  the coefficient functions $b$ and $\sigma$. That is, we approximate the solution $X^{\hat{H}_m}$ by a process $(X^n_t; t\in[0,T])$ given as the solution to
\begin{equation}
\label{eq:coefficient-approx}
    X_t^{n,\hat{H}_M} = x_0 + \int_0^t b_n(s,X^{n,\hat{H}_M}_s)ds +\int_0^t \sigma_n(s,X^{n,\hat{H}_M}_s)dB^{\hat{H}_M}_s.
\end{equation}
Here $n$ denotes the ''approximation-level'' on the coefficient functions, and in general does not necessarily depend on $M$ that corresponds to the estimation of the Hurst index $H$.

Finally, one needs to approximate the continuous trajectories via discretisation. For this, a classical Euler-Maryama scheme for solving \eqref{eq:true_model} is
$$
X_{k+1}= X_k + b(t_k,X_k)\Delta t_k + \sigma(t_k,X_k)\Delta B^H_k.
$$
Note that taking $\sigma(t_k,X_k) = \sigma_{X_0}$ constant given initial data $X_0$ and formally plugging in $H=\frac12$ would give 
$$
X_{k+1}= X_k + b(t_k,X_k)\Delta t_k + \sigma_{X_0}\Delta W_k,
$$
where $W$ is a standard Brownian motion. Coefficient approximation errors in such model are studied, e.g., in \cite{yang2023neural}. In our case, we use continuous interpolation of the Euler-Maruyama scheme and approximate \eqref{eq:coefficient-approx} with
\[X_t^{n,\hat{H}_M,\Delta t} = x_0 + \int_0^t b_n(\eta(s),X^{n,\hat{H}_M,\Delta t}_{\eta(s)})ds +\int_0^t \sigma_n(\eta(s),X^{n,\Delta t,\hat{H}_M}_{\eta(s)})dB^{\hat{H}_M}_s,\]
where $\eta(s) = \Delta t\lfloor\tfrac{s}{\Delta t}\rfloor$.

Our aim is to quantify the triple-error arising from the approximation of $X^H$ with $X^{n,\hat{H}_M,\Delta t}$. That is, we quantify
\begin{align*} &\|X^{n,\hat{H}_M,\Delta t}-{X}^H\|_{\alpha,\infty} \\
&\qquad \leq \| X^{\hat{H}_M} - {X}^H\Vert_{\alpha,\infty} + 
\|{X}_{t}^{\hat{H}_M} - X_{t}^{n,\hat{H}_M}\|_{\alpha,\infty} + \|{X}^{n,\hat{H}_M} - X^{n,\hat{H}_M,\Delta t}\|_{\alpha,\infty} \\
    &\qquad =:\mathcal{E}_{\texttt{fit},\alpha}(M)+\mathcal{E}_{\texttt{appr},\alpha}(n) + \mathcal{E}_{\texttt{time},\alpha}(\Delta t)  ,
\end{align*}
where $
\Vert f\Vert_{\alpha,\infty} = \sup_{0\leq t\leq T}\left\{ |f(t)| + \int_0^t \frac{|f(t)-f(s)|}{|t-s|^{\alpha+1}}ds\right\}
$. The first error represents the fitting error, the second is the coefficient approximation error, and the last is the time-discretisation error. The norm $\Vert \cdot\Vert_{\alpha,\infty}$ in which the error is measured is the norm of the Banach space of functions in which the solutions $X$ belong, hence providing a natural candidate measure for the error. Note that the norm is a fractional Sobolev-type norm involving $|f(t)|$ measuring the size of the function and $\int_0^t \frac{|f(t)-f(s)|}{|t-s|^{\alpha+1}}ds$ measuring the fluctuations of the function.

\subsection{Notation and assumptions}
\label{sec:assumptions}
For a compact $\mathcal{K}\subset \mathbb{R}^d$, let $f:[0,T] \times \mathcal{K} \mapsto \mathbb{R}$. For the norm of $f\in L^\rho(t;L^q(\mathcal{K}))$ we use short notation $\Vert \cdot\Vert_{(\rho,q)}$, i.e. 
$$
\Vert f\Vert_{(\rho,q)} = \left(\int_0^T \left(\int_{\mathcal{K}} |f(t,x)|^qdx\right)^{\frac{\rho}{q}}dt\right)^{1/\rho}.
$$
In particular, throughout the article the norms are considered on a large compact set $\mathcal{K}$ containing all the paths, while this choice of $\mathcal{K}$ will be omitted on the notation.
For $f:[0,T]\mapsto \mathbb{R}$ set
$$
\Vert f \Vert_{\alpha,\lambda} = \sup_{0\leq t\leq T}e^{-\lambda t}\left(|f(t)|+\int_0^t \frac{|f(t)-f(s)|}{(t-s)^{\alpha+1}}ds\right),
$$
$$
\Vert f \Vert_{\alpha,\infty} = \sup_{0\leq t\leq T}\left(|f(t)|+\int_0^t \frac{|f(t)-f(s)|}{(t-s)^{\alpha+1}}ds\right),
$$
and 
$$
\Vert f \Vert_{1,1-\alpha} = \sup_{0\leq s\leq t\leq T}\left(\frac{|f(t)-f(s)|}{|t-s|^{1-\alpha}}+\int_s^t \frac{|f(t)-f(r)|}{(t-r)^{2-\alpha}}ds\right).
$$
We pose the following assumptions that ensure \eqref{eq:true_model} to have a unique solution that is $(1-\alpha)$-H\"older continuous, see \cite[Th. 2.1]{Nualart-Rascanu2002}. On top of that, they allow us to derive bounds for the approximation error, see Theorem \ref{thm:main} below.

\begin{assump}{H0}
The true Hurst parameter $H$ lies in the interval $(\tfrac12, 1)$.
\end{assump}

\begin{assump}{H1} The map ${\bf x}\mapsto \sigma(t,{\bf x})$ is differentiable in $x$ and:
\begin{itemize}
    \item[(H2.a)]\label{H1a} For all $t\in[0,T]$,  $x\mapsto \sigma(t,x)$ is globally Lipschitz continuous. 
    \item[(H2.b)]\label{H1b} For all $t\in[0,T]$, $x\mapsto \nabla\sigma(t,x)$ is locally Lipschitz continuous. 
    \item[(H2.c)]\label{H1c} For all $x\in\mathbb{R}^d$, $t\mapsto \sigma(t,x),\nabla\sigma(t,x)$ are Lipschitz continuous. 
\end{itemize}
\end{assump}

\begin{assump}{H2} The map ${\bf x}\mapsto b(t,{\bf x})$ satisfies the following: 
\begin{itemize}
    \item[(H3.a)]\label{H2a} For all $t\in[0,T]$, $x\mapsto b(t,x)$ is locally Lipschitz continuous.
    \item[(H3.b)]\label{H2b}  There exists a function $b_0\in L^{\infty}(0,T;\mathbb{R}_+)$ and a non-negative constant $L_0$ such that 
    \[|b(t,x)|\le L_0|x|+b_0(t),\forall (t,x)\in\mathbb{R}_+\times\mathbb{R}^d.\]
\end{itemize}
\end{assump}

\begin{assump}{H3} There exists $\gamma\in[0,1]$ and a non-negative constant $K_0$ such that, for all $(t,x)\in\mathbb{R}_+\times\mathbb{R}^d,$
\[|\sigma(t,x)|\le K_0 (1+|x|^\gamma).\]
\end{assump}

\begin{assump}{H4}
The map $t\mapsto b(t,x)$ is $\theta$-H\"older continuous with $\theta\in(2H-1,1]$.
\end{assump}

Since we approximate unknown coefficient functions $b$ and $\sigma$ by $b_n$ and $\sigma_n$, the above requirements are assumed to remain valid for the approximations $b_n$ and $\sigma_n$ as well. Thus we pose the following additional assumption.

\begin{assump}{H5}
Hypotheses \ref{H1}-\ref{H4} are valid for sequences $b_n$ and $\sigma_n$ with Lipschitz and H\"older constants uniformly bounded in $n\in \mathbb{N}$.
\end{assump}

Assumption \ref{H5} is reasonable, as $b_n$ and $\sigma_n$ are supposed to approximate true functions $b$ and $\sigma$. As such, it is natural that the associated constants remain bounded. We stress however, that we do not assume pointwise convergence of $b_n$ and $\sigma_n$ (or their partials): we merely assume the approximations to be reasonable in the sense that the constants remain bounded. Once this is established, we show that weaker $L^p$-type convergence is sufficient for the approximation to converge to the true solution, see Theorem \ref{thm:main} below. As a consequence, one can even consider approximations for which Lipschitz and H\"older constants are not uniformly bounded in $n\in \mathbb{N}$. In this case, one can introduce a cutting that forces the constants to remain bounded, and still obtain small error in the $L^p$-norms.

\begin{remark}\label{rem:H_bounds}
From Assumption  \ref{H0}, we can assume there exists an interval $(\overline{H},\underline{H})\subset (\frac12,1)$ such that $H\in (\overline{H},\underline{H})$. When necessary, we use these bounds without redefining the parameters $\underline{H},\overline{H}$. 
\end{remark}

\subsection{Formulation of the main results}
Our main result is the following. It follows directly from Propositions \ref{prop:time_error}, \ref{prop:fitting_error},  and \ref{prop:coefficient-approximation} studying three different errors. These propositions and their proofs are postponed to the Appendix \ref{sec:proofs}. 
\begin{theorem}
\label{thm:main}
    Suppose that the assumptions of Section \ref{sec:assumptions} hold, in particular, \ref{H0} holds for both $H$ and $\HeM$, with $\hat{H}_M$ an estimator of the Hurst parameter. Let $\epsilon,\epsilon_0>0$, $s,s_2,\tilde{\delta}\in(0,1)$ with $\tilde{\delta}<\HeM$, and denote $\tilde\sigma = \sigma - \sigma_n$ and $\tilde b = b - b_n$. Then, for any 
$q>\frac{d}{1-s}$, $q_2>\frac{1}{1-s_2}, \rho > \frac{1}{1-\alpha}$ such that $q\ge\rho$, and any $\alpha \in (1-H\wedge\HeM, \, \min(1/2,s\tilde{\delta},s_2))$, we have for almost all realisations:
    \begin{align*}
        \mathcal{E}_{\texttt{appr},\alpha}(n) &\leq  C_\omega\left[ \Vert\partial_t\nabla_z\tilde\sigma\Vert_{(q_2,q)} +\Vert\nabla_z\tilde\sigma\Vert_{(q_2,q)} + \Vert\partial_t\tilde\sigma\Vert_{(q_2,q)} + \Vert\tilde\sigma\Vert_{(q_2,q)}\right.  \\
    &+\left.\Vert \nabla_z \tilde b\Vert_{(\rho,q)}  + \Vert  \tilde b\Vert_{(\rho,q)}\right]\\
        \mathcal{E}_{\texttt{fit},\alpha}(M) &\leq C_\omega|\HeM-H|^{\frac12}.
    \end{align*}

    Furthermore, for any $\eta>0$ there exists $\Delta_0 >0$ and $\Omega_0\subseteq \Omega$ such that $\P(\Omega_0)>1-\eta$ and, for any $\Delta t\leq \Delta_0$:
    \begin{align*}
        \mathcal{E}_{\texttt{time},\alpha}(\Delta t) & \leq C_{\epsilon_0}\,  \Delta t^{2\hat{H}_M-1-\epsilon_0}.
    \end{align*}
\end{theorem}
To the best of our knowledge, similar approximation and fitting errors are only presented in \cite{huNu2010} in which simpler time-independent SDEs without drift term were considered, see \cite[Theorem 4]{huNu2010}. The bound in \cite{huNu2010} provides upper bound in this special case in terms of $L^\infty$ norms. In comparison, by using clever potential-theoretic arguments we are able to provide weaker $L^p$-estimates from which one can easily deduce \cite[Theorem 4]{huNu2010} as a special case. Indeed, if one has pointwise convergence for $b_n$ and $\sigma_n$ and their partials, then the approximation error converge to zero automatically. 
The main benefit of our approach is that estimates that convergence in $L^\infty$-norms are much harder to obtain compared to weaker $L^p$-type norms where pointwise convergence does not necessarily hold. This provides an advantage as, for example, typical kernel based methods usually provide convergence in $L^p$-spaces but not in the pointwise sense. As a concrete application, see discussion on approximating with shallow neural networks in Section \ref{sec:ML_application}. We also stress that our upper bound does not depend on the solution paths, as long as they can be contained inside a large compact set $\mathcal{K}$ which is possible whenever hypothesis \ref{H5} holds. Finally, we note that the approximation error presented in Theorem \ref{thm:main} remains valid regardless of the noise $G$ as long as the noise is H\"older continuous of sufficient order. For the fitting error in the more general case, one obtain a bound directly from \cite[Theorem 2]{Viitasaari-2022} in terms of the error made in the noise. The above formulation for the fitting error arises in particular from the estimation of the Hurst index $H$, see Lemma \ref{lem:brownian_convergence}.
\begin{remark}
\label{rem:sobolev}
By using H\"older inequality, note that we can bound the approximation error by 
\begin{align*}
        \mathcal{E}_{\texttt{appr},\alpha}(n) &\leq  C\left[ \Vert \tilde\sigma\Vert_{W^{2,q_2q}} + \Vert \tilde{b}\Vert_{W^{1,\rho q}}\right], 
    \end{align*}
    where $W^{k,p}$ denotes the Sobolev space on $[0,T]\times \mathcal{K}$. This provides a natural norm where errors $\tilde{\sigma}$ and $\tilde{b}$ should be measured.
    \end{remark}
\begin{remark}
\label{remark:fitting-error}
    The fitting error in Theorem \ref{thm:main} depends explicitly on the choice of estimator $\hat{H}_M$, and convergence rates for the approximation can be deduced accordingly. For instance, by considering the one-dimensional case we can use the consistent estimator proposed in \cite{Kubilius-Mishura-Ralchenko2017} given by
\begin{align}\label{eq:H_estimator}
\hat{H}_M &= \frac12 - \frac1{2\log2}\log\left(\frac{\sum_{m=1}^{2M-1}(\hat{X}_{t_{m+1}^{2M}} - 2\hat{X}_{t_{m}^{2M}}+\hat{X}_{t_{m-1}^{2M}})^2}{\sum_{m=1}^{M-1}(\hat{X}_{t_{m+1}^{M}} - 2\hat{X}_{t_{m}^{M}}+\hat{X}_{t_{m-1}^{M}})^2}\right),
\end{align}
where $\Pi^M=\{t_0^M,\ldots,t_M^M\}$ is a partition of size $M$ of the interval $[0,T]$. Then under the assumptions of \cite[Theorem 3.6]{Kubilius-Mishura-Ralchenko2017}, and the restriction of $\HeM$ to $(\frac12,1)$ if necessary, we get 
\[ \hat{H}_M = H+ O\left(\left(\frac{\log(M)}{M}\right)^{\gamma/2}\right),\]
for any $\gamma\in (1/2, H)$ leading to
\begin{equation*}
\mathcal{E}_{\texttt{fit},\alpha}(M) \leq C\left(\frac{\log(M)}{M}\right)^{\gamma/4}.
\end{equation*}
\end{remark}

 \begin{remark}
 From Lemma \ref{lemma:uniform-boundedness} and assumptions \ref{H1}-\ref{H2} we can apply Rademacher Theorem and guarantee existence of the weak derivatives for Theorem \ref{thm:main}. On the other hand, assumption $b_0\in L^\infty(0,T;\mathbb{R}_+)$, can be generalized. However, this paper concentrates on examining the errors arising from the approximation of the coefficients in Equation \eqref{eq:true_model}, while also taking the time-discretization and fitting errors into account. Therefore, we focus on applying existing results from the literature related to time approximation errors within our framework, without worrying about potential extensions towards this direction.
\end{remark}
\begin{remark}
Notice that different time-approximations can be considered, for instance, Milstein-type schemes with rate $\Delta t^{H-\alpha}\sqrt{-\log(\Delta t)}$ according to the results in \cite{Deya-Neuenkirch-Tindel-2012}, under the stronger condition $\sigma_n,b_n\in C^3_b(\mathbb{R}^d)$. Under similar assumptions, in \cite{Hu-Liu-Nualart} the authors propose some generalizations of Euler-scheme relying on the smoothness of the diffusion coefficient, with $L^p$ strong rate of convergence  $\Delta t^{2H - 1/2}$ if $H\in(1/2,3/4)$ and $\Delta t$ if $H\in(3/4,1)$. However, we aim to adhere to the regularity hypotheses proposed in~\cite{Nualart-Rascanu2002} to ensure well-posedness and solution regularity. 
\end{remark}

\section{Application to Quantitative Approximation by Neural Networks}\label{sec:ML_application}

In the context of neural networks, we can apply our results by incorporating universal approximation theorems. These theorems demonstrate that a feedforward neural network with at least one hidden layer can approximate any continuous function on a closed interval to any desired degree of accuracy, provided it has a sufficient number of neurons. 

For a given realization of the noise, let us consider $\mathcal{K}\subset\R^{d+1}$ be a compact large enough such that it supports the trajectory of the solution $X_t^{n,\hat{H}_M}$ to \eqref{eq:coefficient-approx}, with $b_n$ (respectively $\sigma_n$) as the approximation of $b$  (respectively $\sigma$) by means of shallow neural networks with activation function $\phi$. 

Following Siegel and Xu~\cite{siegel2024}, we approximate $b$ and $\sigma$ (living in some subspace of ${L}^{p}(K;\R^d)$ for some $p\geq1$) through linear combination of elements in some uniformly bounded dictionary $D_\phi = \{\phi(\langle\theta,\cdot\rangle):~ \theta\in \R^{d+1}\}\subset {L}^{p}(K;\R)$. More precisely, we approximate $b,\sigma$ with elements in 
\begin{align*}
\Sigma_{n,m}(\phi) := \left\{\sum_{j=1}^n \omega_j \phi_j:\, \phi_j\in D_\phi,\,  \omega_j\in \R,\, \sum_{j=1}^n|\omega_j|\leq m\right\},
\end{align*}
which corresponds to the set of single-hidden-layer neural networks with $n$ units, activation function $\phi$, and output weights bounded by $m$. 

In shallow neural networks, controlling activation growth is essential for stability. Bounded activations ensure Lipschitz continuity, prevent gradient explosion, and satisfy classical approximation theorems (see, e.g. \cite{hornik1994}). Unbounded activations (e.g., ReLU) are particularly useful to construct regressors for certain types of non-smooth targets \cite{yarotsky2017} and handle simple kinks and ridges effectively. Nevertheless, it cannot compactly represent arbitrary non-smooth features such as complex discontinuities or highly oscillatory behavior. In practice, unbounded activations can be suitable regressors but require explicit constraints to avoid unstable growth and ensure generalization. In terms of the dictionary $D_\phi$, $D_\phi$ is uniformly bounded when $\phi$ is bounded, whereas in the more general case of unbounded activation, we need to restrict the set of admissible weights. In particular, when $\phi(x) = \text{ReLU}(x)$, the dictionary $D_k := \left\{(\langle\theta,\cdot \rangle)^k\vee 0:\, \theta\in\mathbb{S}^{d-1}\right\}\subset {L}^p(K)$ is uniformly bounded. Here we are omitting the bias weight for simplicity of notation. However, we remark that for unbounded activations the bias parameters must also be properly clipped, see \cite{siegel2024}. 

The candidates that one wants to approximate are the elements of the closure of the convex hull of the dictionary $D_\phi$, see \cite{hornik1994}, given by
\begin{align*}
B_1(D_\phi):=\overline{\left\{\sum_{j=1}^n \omega_j\phi_j :\, \phi_j\in D_\phi,\, n\in\mathbb{N}, \, \omega_j\in\R,\; \sum_{j=1}^n |\omega_j|\le 1\right\}}.
\end{align*}

Focusing on approximating $(d+1)$-dimensional vector fields $b$ and $\sigma$ for which weak derivatives exists, here we apply the results of Siegel and Xu \cite{siegel2024} to the type-2 Banach space $W^{r,p}(K) $, where $p\geq2$, $K=[0,T]\times [-N,N]^d$, and $N$ is defined as in Lemma \ref{lemma:uniform-boundedness}. 
For the sake of completeness of our presentation, we summarize the result of \cite{siegel2024} applied into our situation. Note that here we leverage key insights from \cite{siegel2024}, while a comprehensive treatment of the techniques falls outside our present scope. Interested readers are directed to original paper and the references therein for complete details regarding approximation properties in a more general setting.

\begin{proposition}[Theorem 1-3, Siegel and Xu\cite{siegel2024}]\label{prop:siegel2024}
Let $\phi$ be a bounded activation function and $D_\phi\subset {L}^p([0,T]\times[-N,N]^d)$ the corresponding uniformly bounded dictionary. Then for $f\in B_1(D_\phi),$ we have
\begin{align}
\inf_{f_n\in \Sigma_{n,1}(\phi)}\Vert f-f_n\Vert_{W^{r,p}(K)}\leq \frac{C}{\sqrt{n}},\end{align}
where the constant $C>0$ depends on $\sup_{d\in D_\phi}\Vert d\Vert_{W^{r,p}(K)}$, $r$, $p$, and $K$, but not on $n.$
Moreover, in the case of $k$-ReLU activation there exists a constant $M = M(p, k, d) > 0$ such that, for all $f \in B_1(D_k)$, we have
\begin{align}\label{eq:estimation_ReLU_nn}
\inf_{f_n \in \Sigma_{n, M}(D_k)} \| f - f_n \|_{{L}^p([0,T]\times[-N,N]^d)} \le C n^{-\frac{1}{2} - \frac{pk+1}{p(d+1)}}.
\end{align}
\end{proposition}

\begin{remark}
Note that when we consider $f\in B_1(D_\phi)$, or $f\in B_1(D_k)$, we ensure the existence of a minimizing neural network. However, this does not guarantee its uniqueness or that a specific training algorithm will allow us to actually obtain such a neural network.    
\end{remark}

If we consider $b_n,\sigma_n\in\Sigma_{n,1}(D_\phi)$, for $\phi$ a bounded activation, assumptions \ref{H1}-\ref{H4} hold true; however, the Lipschitz constant is not bounded uniformly in $n$ unless the input weights $\theta$ are bounded. Taking a closer look at the proof of Theorem 1 in \cite{siegel2024}, it is easy to check that the theorem is valid if we change the dictionary in order to consider only bounded input weights/parameters:
\[D_\phi = \{\phi(\langle\theta,\cdot\rangle):~ \|\theta\|\leq \Theta\}\subset {L}^{p}(K;\R)\]
for some constant $\Theta>0$. Of course in this case the set of target functions changes with the -still uniformly bounded- dictionary.
Thus, with the help of the previous result and our main Theorem \ref{thm:main} we obtain the following corollary.

\begin{corollary}\label{cor:NN_case_rate}
Assume hypotheses of Theorem \ref{thm:main} hold true. Let $\phi$ be a bounded differentiable activation and $D_\phi$ the corresponding uniformly bounded dictionary with bounded input weights. Let $q,q_2,\rho\geq 2$ as in Theorem \ref{thm:main}. If $b,\sigma\in B_1(D_\phi)$, then 
\begin{align*}
\inf_{b_n,\sigma_n\in \Sigma_{n,1}(D_\phi)}\mathcal{E}_{\texttt{appr},\alpha}(n) &\leq  \frac{C(q,q_2,\rho)\, K_\phi}{\sqrt{n}},
\end{align*}
where $K_\phi = \sup_{d\in D_\phi}\Vert d\Vert_{W^{2,pq}(K)}$, with $p=q_2,\rho$, provided that $pq\ge2$.

\end{corollary}
To prove the corollary it is enough to notice that from Theorem \ref{thm:main} (see also Remark \ref{rem:sobolev}) we that
\begin{align*}
\inf_{b_n,\sigma_n\in \Sigma_{n,1}(D_\phi)}\mathcal{E}_{\texttt{appr},\alpha}(n) &\leq  C\inf_{b_n,\sigma_n\in \Sigma_{n,1}(D_\phi)} \left\{\Vert  \tilde \sigma\Vert_{W^{2,q_2 q}(K)}+ \Vert  \tilde b\Vert_{W^{1,\rho q}(K)}\right\}\leq \frac{C\, K_\phi}{\sqrt{n}},
\end{align*}
where the last inequality is obtained from Proposition \ref{prop:siegel2024}.

\begin{remark}
It is evident that when considering shallow neural network approximations for the drift coefficient 
$ b_n \in \Sigma_{n,M}(D_k) $, the assumptions in Section~\ref{sec:assumptions} are in general satisfied (in particular, the Lipschitz constants remain uniformly bounded). Moreover, the differentiability of $b_n$
holds in a weak sense, which is sufficient for the analysis of its Sobolev norm. This observation suggests the possibility of employing a mixed approximation scheme, with bounded activation functions for the diffusion term and ReLU activations for the drift. However, available estimates for ReLU activations are typically established 
in $ {L}^p $-spaces. Consequently, further results involving universal approximation theorems in Sobolev norms become relevant in this setting. 
\end{remark}

\begin{remark}[Approximation by deep neural networks]
In the multilayer case, Yarotsky \cite{yarotsky2017} provided error bounds for deep ReLU networks that approximate functions in Sobolev spaces $W^{r,\infty} $, providing convergence rates measured in the uniform norm ${L}^\infty $. Although the analysis does not cover the Sobolev norms directly, his work laid important groundwork on how network depth enables exponential approximation efficiency for certain function classes, motivating subsequent advances in Sobolev and Besov space approximations.
Later, Siegel \cite{siegel2024} establishes convergence rates for deep ReLU neural networks that target functions within Sobolev spaces, assuming that the bounded domain is the unit cube. For a fixed number of neurons (which depends on the dimension $d$), the results in \cite{siegel2024} provide sharp bounds for the ${L}^p$-norms of the approximation error that explicitly consider smoothness, dimension, and the number of layers. In addition, further results are presented on variations in both the number of layers and the number of units. As such, these works provide grounds for further practical applications of our approach.
\end{remark}

\pagebreak

\section{Numerical Experiments}\label{sec:numerics}

In this section, we conduct numerical experiments to study the 
the practical applicability of the proposed error-analysis framework for nonparametric estimation of fractional stochastic differential equations. Using controlled synthetic datasets, we illustrate how each source of error contributes to the overall error. We also examine the recovery of neural nonparametric estimators for both the drift and diffusion terms when trained solely on discrete-time observations.

In the following, we present the experimental setup, providing details on the data generation process, settings for experiments, and metrics for benchmarking.

\subsection{Data}\label{subsec:setup}

We consider the following one- and two-dimensional examples  
for the true equation
\begin{align*}
\mathrm{d}X_t
&= b(t,X_t)\,\mathrm{d}t
    + \sigma(t,X_t)\,\mathrm{d}B^H_t.
\end{align*}
\paragraph*{\bf 1D model.}
The coefficients functions are given by
\begin{align*}
b(t,x)
&= -x + \tfrac{1}{4}\tanh(x),\\
\sigma(t,x)
&= 0.5 + 0.2\,\tanh(x). 
\end{align*}

\paragraph*{\bf 2D model.}
The coefficient functions  the coefficients are given by
\begin{align*}
b(t,x) &= 
\begin{pmatrix} 
-0.8 & 0 \\ 
0 & -0.4 
\end{pmatrix} 
\begin{pmatrix} 
x_1 \\ 
x_2 
\end{pmatrix} + 
\begin{pmatrix} 
0.25 \sin(2\pi t) \\ 
0.2 \cos(2\pi t) 
\end{pmatrix} \\[10pt]
\sigma(t,x) &= 
\begin{pmatrix}
0.6 + 0.15 \tanh(x_1) & 0 \\
0 & 0.6 + 0.15 \tanh(x_2)
\end{pmatrix}.
\end{align*}

We consider Hurst parameters $H \in \{0.5,\,0.7,\,0.9\}$ to examine also the behaviour in terms of $H$ that controls both the memory effect as well as the path regularity. The case $H=0.5$ is included for comparison with more classical frameworks. 

Initial states are sampled as $X_0 \sim \mathcal{N}(0,I_d)$ and are clipped to $[-N,N]^d$, where $N>0$ is chosen a~priori large enough to contain more than $99.9\%$ of observed states during training. For each configuration, we generate $160$ independent trajectories: $100$ for training, $28$ for validation, and $32$ for testing.

The noise process $B^H$ is simulated on a uniform time grid using the Davies--Harte circulant-embedding method, which yields exact Gaussian increments on uniform meshes and is computationally efficient. For the time discretisation, we employ a coarse observation grid
$t_m = m\,\widehat{\Delta}$ with 
$0\le m \le M$, $M$
being the length of each observed trajectory. 
Within each coarse interval we simulate on a finer mesh with $\Delta t_{\mathrm{fine}} = \widehat{\Delta}/k$, where $k \in \N$ denotes the number of subintervals of each coarse step $\widehat{\Delta}$. Finally, we create data samples by using 
\begin{align}\label{eq:euler-maruyama}
X_{i+1}
&= X_i
 + b(t_i,X_i)\,\Delta t_{\mathrm{fine}}
 + \sigma(t_i,X_i)\,\Delta B^H_i,\\
\Delta B^H_i
&= B^H_{t_{i+1}} - B^H_{t_i},\quad \text{for all }i\in \{0,1,\ldots, kM-1\}. \nonumber
\end{align}
After simulating at the fine scale, we downsample to the coarse grid to obtain the observed trajectories
$
\{\widehat{X}_{t_m} : 0 \le m \le M\}.
$

\subsection{Estimation of $H$}\label{subsubsec:estH}
We estimate $H$ from observations on the coarse grid using a second-order increment ratio estimator on aligned dyadic refinements (see, e.g., \cite{Kubilius-Mishura-Ralchenko2017}). The resulting estimate is denoted by $\widehat{H}_M$ and is clipped to $(\tfrac12,0.99)$ to ensure Assumption \ref{H0}.

\subsection{Estimation of coefficient functions}
For the estimation of the coefficient functions $b$ and $\sigma$, we parameterise $b_n$ and $\sigma_n$ as single-hidden-layer neural networks with $\tanh$ activation and clipped weights and biases to enforce uniform Lipschitz bounds (consistent with Assumptions~\ref{H1}--\ref{H5}). For input $(t,x)\in[0,T]\times\R^d$ and hidden width $n$, set

\begin{align}
b_n(t,x)
&= W_2\,\phi\!\bigl(W_1(t,x)^\top + b_1\bigr) + b_2,\label{eq:bn-param}\\
\sigma_n(t,x)
&= \widetilde{W}_2\,\phi\!\bigl(\widetilde{W}_1(t,x)^\top + \widetilde{b}_1\bigr) + \widetilde{b}_2,\label{eq:sigman-param}
\end{align}
where $\phi = \tanh$. In the case $d=2$, $\sigma_n$ outputs the two diagonal entries of a diffusion matrix, positivity is enforced by a softplus transform or a positive clamping.

\textbf{Loss function.} Denote by $\{{X}_{t_m} : 0 \le m \le M\}$ and $\{\widehat{X}_{t_m} : 0 \le m \le M\}$ the realized and estimated trajectories, respectively. Discretised version of the fractional Sobolev-type norm for the error between simulated and realized trajectories is derived using
\begin{align}
\label{eq:frac-norm-discrete}
\norm{f}_{\alpha,\infty}&\approx\; \max_{0\le m\le M}\!\left(
\norm{f(t_m)} + \sum_{k=0}^{m-1}
\frac{\norm{f(t_m)-f(t_k)}}{\big((m-k)\,\widehat{\Delta}\big)^{\alpha+1}}\,\widehat{\Delta}
\right),\quad\\
f(t_m)&=\widehat{X}_{t_m}-X_{t_m},
\end{align}
where $\alpha\in(1-H,1/2)$. Intuitively, the first term, $ \norm{f(t_m)}$, captures the difference in drift/trend behavior while the second term $\sum_{k=0}^{m-1}
\frac{\norm{f(t_m)-f(t_k)}}{\big((m-k)\,\widehat{\Delta}\big)^{\alpha+1}}\,\widehat{\Delta}$ expresses the error related to the diffusion term by evaluating the weighted average of difference between the 
dissimilarity of two trajectories for different time lengths. Computing the expectation of the fractional Sobolev-type norm of the error, we use the loss function 

\begin{align}
\label{eq:expected-frac-norm}
\mathcal{L}_{}(\theta)=
\mathbb{E}\!\Bigg[
\max_{0\le m\le M}
\Bigg(
\|f(t_m)\|
+
\sum_{k=0}^{m-1}
\frac{\|f(t_m)-f(t_k)\|}
{\big((m-k)\,\widehat{\Delta}\big)^{\alpha+1}}
\,\widehat{\Delta}
\Bigg)
\Bigg].
\end{align}
to compute the gradients and update the parameters of the networks.

\textbf{Optimisation.}
We train using Adam optimizer with weight decay. Early stopping is based on validation of the fractional loss. Each batch consists of full trajectories. Gradient accumulation is used across multiple trajectories for memory efficiency. The hyperparameters are selected by validation. 
The full training procedure is presented in Algorithm \ref{alg:ours-opt}.

\begin{algorithm}[t]
\setlength{\parindent}{0pt}
\raggedright

\textbf{Input:}
 $\{X^{(j)}_{t_0},\ldots,X^{(j)}_{t_M}\}_{j=1}^N$, realized trajectories of a stochastic process until the time horizon $t_M$ with $M$ being the sample size and N being the number of trajectories,
coarse step $\widehat{\Delta}$,
fine step $\Delta t_{\mathrm{fine}}$,
learning rate $\eta$,
Hurst index $H$,
fractional order $\alpha\in(0,1/2)$.

\begin{algorithmic}
\STATE $L \leftarrow \widehat{\Delta}/\Delta t_{\mathrm{fine}} $
\end{algorithmic}

\textbf{Output:}
Optimized parameters $(\theta_b,\theta_\sigma)$ for drift and diffusion.

\textbf{Initialize:}
$\theta \leftarrow (\theta_b,\theta_\sigma)$.

\begin{algorithmic}
\WHILE{not converged}
        \STATE $\mathcal{L}(\theta) \leftarrow 0$.
        \FOR{$j = 1,\ldots,N$}
            \STATE Simulate $\{\hat X^{(j)}_{t_{m,\ell}}\}_{m=0,\ldots,M;\,\ell=0,\ldots,L-1}$ on the fine grid with $\theta_b$, $\theta_\sigma$.
            \STATE Form path error $f^{(j)}(t_m) \leftarrow \hat X^{(j)}_{t_{m,L-1}}\ - X^{(j)}_{t_m}$.
            \STATE Compute fractional path loss $\|f\|^{(j)}_{\alpha,\infty}$
             using the loss formula in \ref{eq:frac-norm-discrete} as
            \STATE $\mathcal{L}(\theta) \leftarrow \mathcal{L}(\theta) + \|f\|^{(j)}_{\alpha,\infty}$.
        \ENDFOR
        \STATE $\mathcal{L}(\theta) \leftarrow \mathcal{L}(\theta)/N$.
        \STATE $\theta \leftarrow \theta - \eta \nabla_\theta \mathcal{L}(\theta)$.
\ENDWHILE
\end{algorithmic}

\caption{Learning drift and diffusion via fractional path loss}
\label{alg:ours-opt}
\end{algorithm}

\subsection{Metrics and Evaluation}
\label{subsec:metrics}

We evaluate parameter recovery of coefficients and dynamical fidelity on held-out test trajectories. To assess the accuracy of the estimated coefficients we report empirical $L^2$ and relative $L_{rel}^2 $ errors for drift and diffusion, under uniform sampling scheme over $K=[0,T]\times[-N,N]^d$. That is, we consider

\begin{align*}
L^{2}(b) &= \sqrt{\frac{1}{M} \sum_{i=1}^{M} \|b_n(t_i, X_{t_i}) - b(t_i, X_{t_i})\|^2}, \\
L^{2}(\sigma) &= \sqrt{\frac{1}{M} \sum_{i=1}^{M} \|\sigma_n(t_i, X_{t_i}) - \sigma(t_i, X_{t_i})\|^2},
\end{align*}

\begin{align*}
L^{2}_{\text{rel}}(b) &= \sqrt{\frac{1}{M} \sum_{i=1}^{M} \left( \frac{\|b_n(t_i, X_{t_i}) - b(t_i, X_{t_i})\|}{\|b(t_i, X_{t_i})\|} \right)^2}, \\[10pt]
L^{2}_{\text{rel}}(\sigma) &= \sqrt{\frac{1}{M} \sum_{i=1}^{M} \left( \frac{\|\sigma_n(t_i, X_{t_i}) - \sigma(t_i, X_{t_i})\|}{\|\sigma(t_i, X_{t_i})\|} \right)^2}.
\end{align*}
We measure the fractional Sobolev norm error, \eqref{eq:frac-norm-discrete}, across multiple test trajectories and report the mean and standard deviation.

\subsection{Results}
\label{subsec:mainresults}

\paragraph*{\bf Function Recovery}

Figure~\ref{fig:heatmaps_1d} reports the one-dimensional results for Hurst indices $H\in\{0.5,0.7,0.9\}$. Comparing the estimated drift and diffusion functions in Figure~\ref{fig:heatmaps_1d}b(f), Figure~\ref{fig:heatmaps_1d}c(g), and Figure~\ref{fig:heatmaps_1d}d(h) (corresponding to $H=0.5$, $0.7$, and $0.9$, respectively) against the ground-truth coefficients in Figure~\ref{fig:heatmaps_1d}a(e) demonstrates that the coefficients are recovered accurately across all three settings.

Table~\ref{tab:part1_hursts2_hidden128_rel} complements this qualitative evidence by aggregating performance over multiple observed trajectories, thereby assessing robustness to randomness. As the Hurst index increases, sample paths become smoother, which typically yields smaller fractional path errors and more stable drift estimates. At the same time, the reduced stochastic variability makes the diffusion estimation task less informative: increased temporal regularity lowers path roughness and can therefore lead to weaker learning signals for the diffusion network.

\begin{figure}[h]
\centering
\includegraphics[width=\textwidth,height=0.36\textheight]{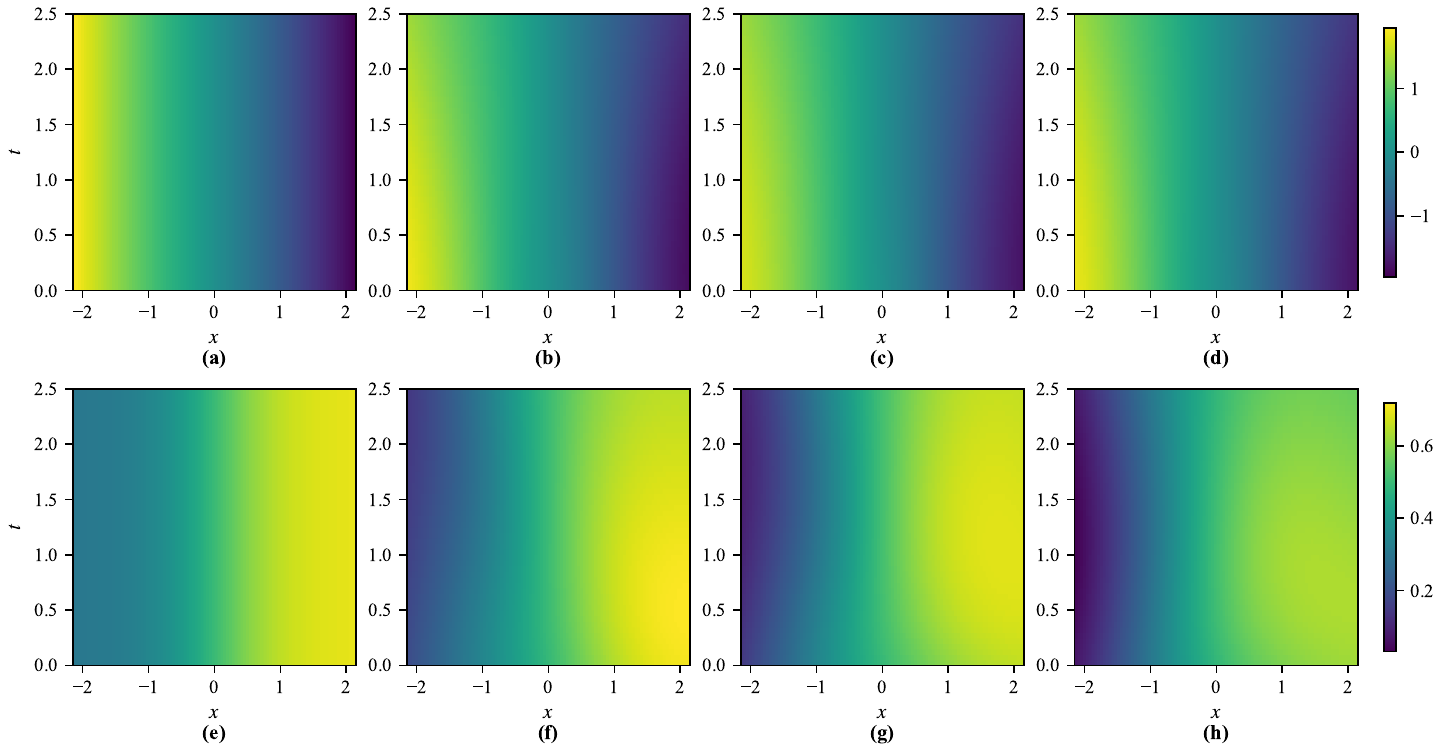}
\caption{The drift (first row) and diffusion (second row) values for both scalar fields: the ground truth (first column) and their estimates in the 1D case. 
The second, third and fourth columns correspond to estimation for Hurst exponents of $0.5$, $0.7$ and $0.9$, respectively.}
\label{fig:heatmaps_1d}
\end{figure}

\begin{table}[htbp]
    \centering
    \renewcommand{\arraystretch}{1.4}
    \caption{Mean and standard deviations of fractional Sobolev
norm and drift/diffusion $L^2$ and $L_{rel}^2 $ errors in 1D settings with $\widehat{\Delta}=0.05$, $\Delta t_{\mathrm{fine}}=\widehat{\Delta}/4$, and hidden width $n=128$.}
    \vspace{0.7em}
    {\footnotesize
    \begin{tabularx}{0.95\textwidth}{r *{5}{>{\centering\arraybackslash}X}}
        \hline
        \textbf{$H$} &
        \textbf{$\norm{f}_{\alpha,\infty}$} &
        \textbf{$L^{2}(b)$} &
        \textbf{$L^{2}(\sigma)$} &
        \textbf{$L^{2}_{rel.}(b)$} &
        \textbf{$L^{2}_{rel.}(\sigma)$} \\
        \hline
        0.5 & $0.0424 \pm 0.0150$ & $0.0136 \pm 0.0061$ & $0.0014 \pm 0.0004$ & $0.0129 \pm 0.0058$ & $0.0050 \pm 0.0016$ \\
        0.7 & $0.0329 \pm 0.0106$ & $0.0185 \pm 0.0038$ & $0.0020 \pm 0.0010$ & $0.0177 \pm 0.0036$ & $0.0073 \pm 0.0038$ \\
        0.9 & $0.0184 \pm 0.0059$ & $0.0157 \pm 0.0040$ & $0.0043 \pm 0.0023$ & $0.0150 \pm 0.0039$ & $0.0159 \pm 0.0084$ \\
        \hline
    \end{tabularx}
    }
    \label{tab:part1_hursts2_hidden128_rel}
\end{table}

Applying the same training framework in the two-dimensional setting yields analogous behavior. Table~\ref{tab:part1_hursts2_hidden128_rel_2d} reports the mean performance and standard deviation computed over multiple optimizations with different initial values, while 
Figure~\ref{fig:heatmaps_2d} depicts one of the estimated drift and diffusion coefficients for different Hurst indices, alongside the ground truth.

Figure~\ref{fig:heatmaps_2d} indicates that diffusion estimation becomes less accurate as the state moves away from the origin: in these regions the drift term dominates the dynamics, making diffusion-related weight updates less effective. This phenomenon is further highlighted in Figure~\ref{fig:trajs_2d}, which compares true and estimated trajectories for a fixed Brownian path. In particular, near the origin, smaller Hurst values produce more diffusion-driven behavior and stronger oscillations, providing a richer learning signal and resulting in more accurate diffusion estimates.

\begin{figure}[h]
\centering
\includegraphics[width=\textwidth,height=0.36\textheight]{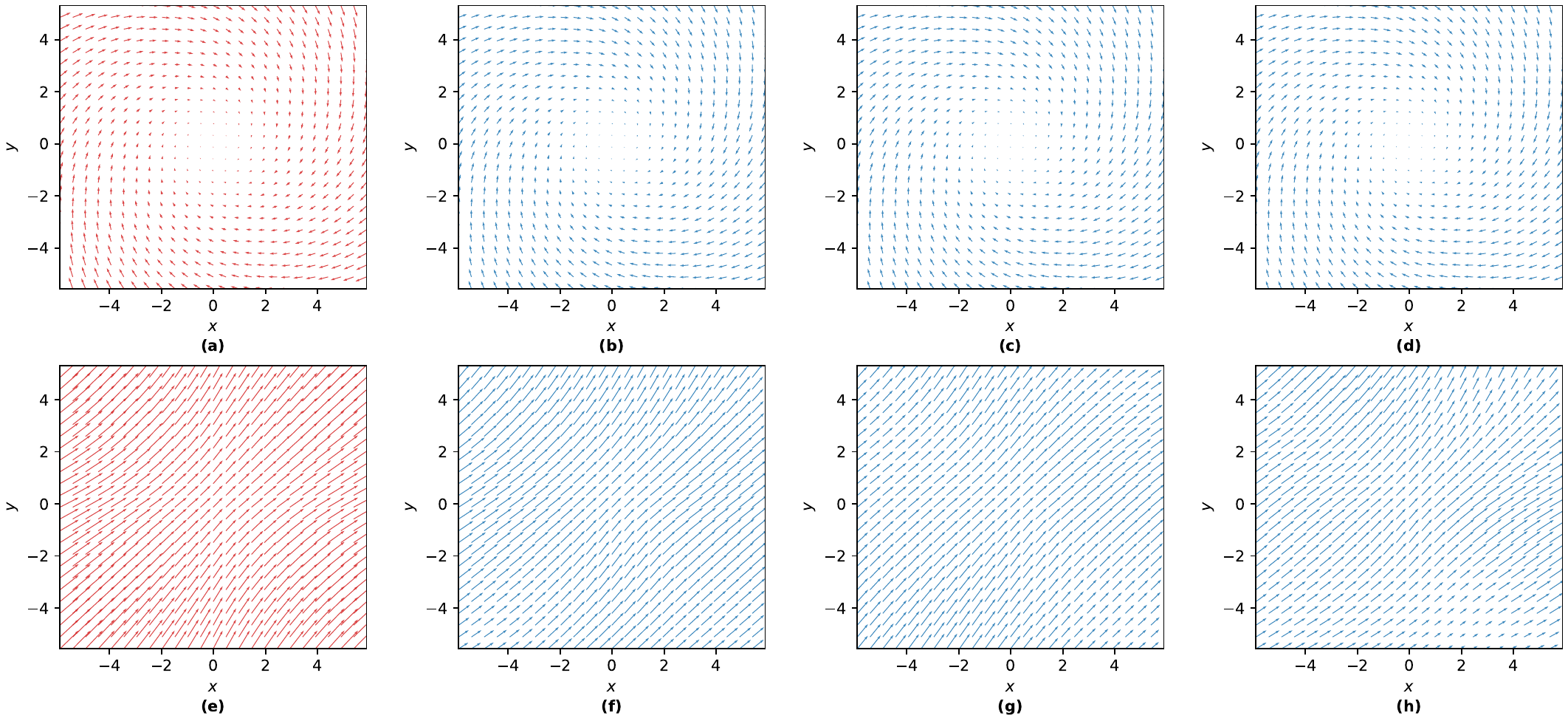}
\caption{The drift (first row) and diffusion (second row) values for both vector fields: the ground truth (first column) and their corresponding estimates in the 2D case. The second, third and fourth columns correspond to estimation for Hurst exponents of $0.5$, $0.7$ and $0.9$, respectively.}
\label{fig:heatmaps_2d}
\end{figure}

\begin{figure}[h]
\centering
\includegraphics[width=\textwidth,height=0.46\textheight]{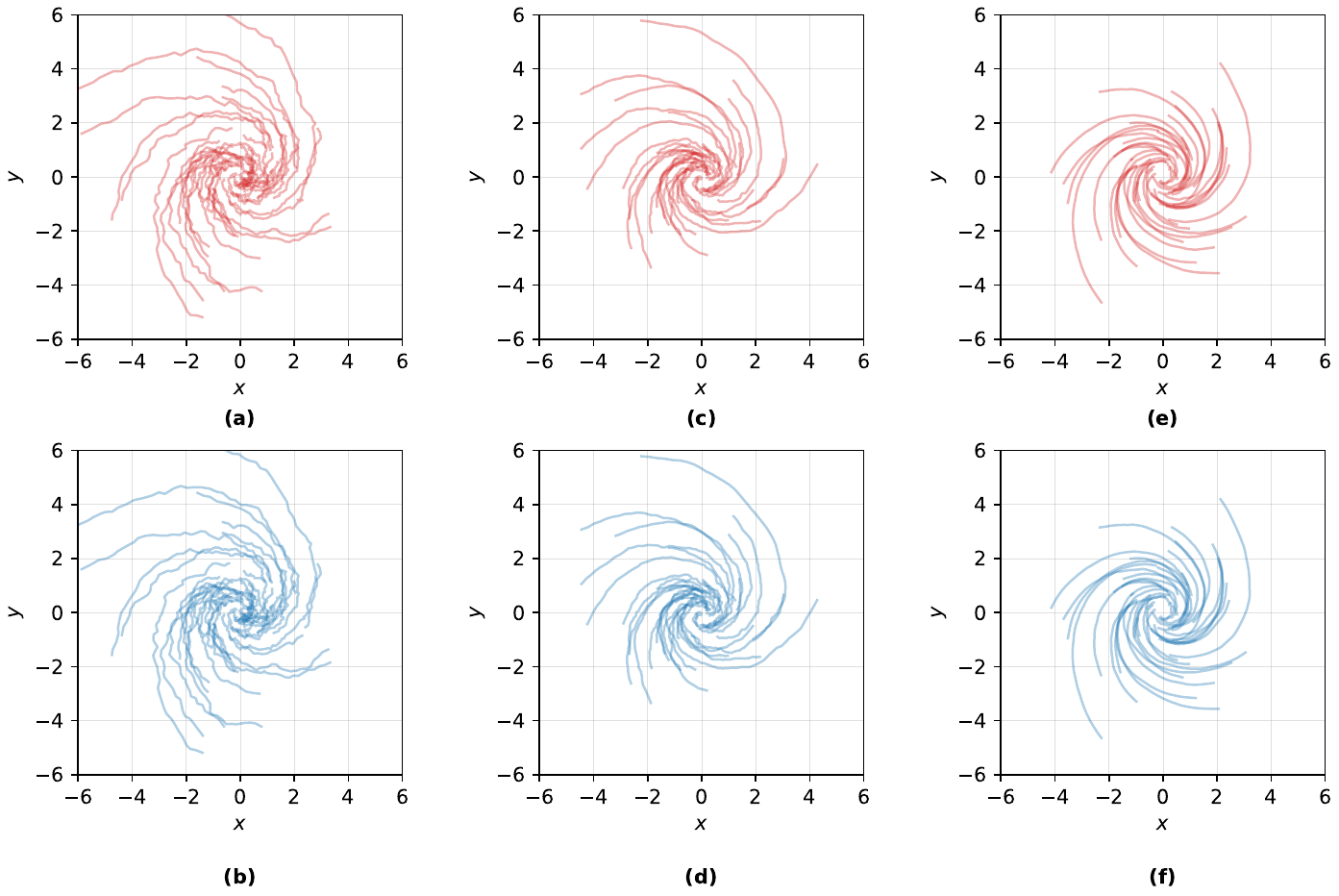}
\caption{Two-dimensional trajectories for the ground truth (top row) and the estimated models (bottom row). Columns correspond to $H=0.5$, $H=0.7$ and $H=0.9$, respectively.}
\label{fig:trajs_2d}
\end{figure}

\begin{table}[htbp]
    \centering
    \renewcommand{\arraystretch}{1.4}
    \caption{
Mean and standard deviations of fractional Sobolev
norm and drift/diffusion $L^2$ and $L_{rel}^2 $ errors in 2D settings with
    $\widehat{\Delta}=0.05$, $\Delta t_{\mathrm{fine}}=\widehat{\Delta}/4$, and hidden width $n=256$.}
    \vspace{0.7em}
    {\footnotesize
    \begin{tabularx}{0.95\textwidth}{r *{5}{>{\centering\arraybackslash}X}}
        \hline
        \textbf{$H$} &
        \textbf{$\norm{f}_{\alpha,\infty}$} &
        \textbf{$L^{2}(b)$} &
        \textbf{$L^{2}(\sigma)$} &
        \textbf{$L^{2}_{rel.}(b)$} &
        \textbf{$L^{2}_{rel.}(\sigma)$} \\
        \hline
        0.5 & $0.6367 \pm 0.2469$ & $1.4088 \pm 0.3531$ & $0.0118 \pm 0.0031$ & $0.0980 \pm 0.0116$ & $0.0572 \pm 0.0139$ \\
        0.7 & $0.4933 \pm 0.2426$ & $1.4156 \pm 0.3563$ & $0.0131 \pm 0.0015$ & $0.0995 \pm 0.0125$ & $0.0650 \pm 0.0110$ \\
        0.9 & $0.3897 \pm 0.1301$ & $1.3790 \pm 0.2297$ & $0.0142 \pm 0.0022$ & $0.0985 \pm 0.0055$ & $0.0695 \pm 0.0083$ \\
        \hline
    \end{tabularx}
    }
    \label{tab:part1_hursts2_hidden128_rel_2d}
\end{table}

\paragraph*{\bf Empirical Error Decomposition}

As shown in our main result (Theorem~\ref{thm:main}), the overall approximation error admits a decomposition into three components: the fitting error, the time-discretization error, and the coefficient-approximation error. In this section, we present a set of experiments on one-dimensional data to illustrate these contributions. During the whole experiments in this section, we fix the fine time step to $\Delta t_{\mathrm{fine}} = 0.05$, $\Delta t_{\mathrm{fine}}=\Delta t_{\mathrm{coarse}}/4$ and use the oracle value $H = 0.7$.

\medskip
 For shallow networks with bounded activation functions, Corollary~\ref{cor:NN_case_rate} predicts a decay rate of order $n^{-1/2}$, which is consistent with the trend observed in Figure~\ref{fig:frac_error_vs_neurons}. Table~\ref{tab:nets_hurst07} further confirms this behavior: wider networks achieve smaller fractional path errors, reflecting more accurate approximation of the drift and diffusion coefficients.

\begin{figure}[h]
\centering
\includegraphics[width=0.85\textwidth,height=0.36\textheight]{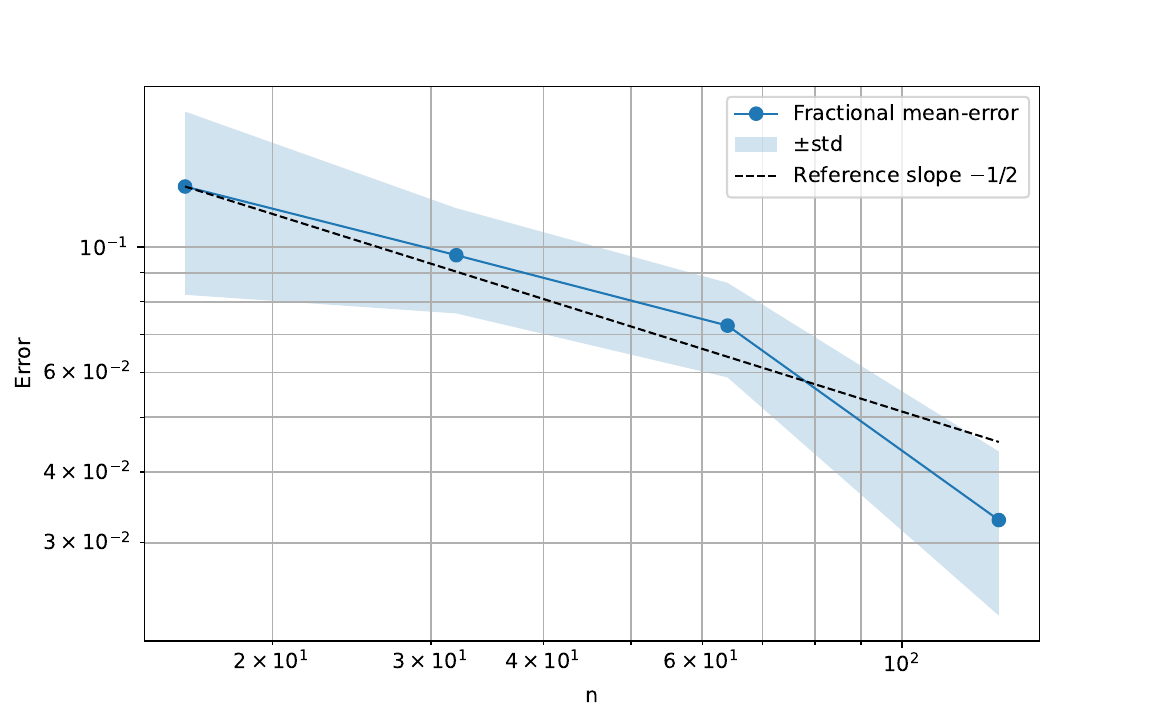}
\caption{Mean and standard deviations of validation losses versus hidden-layer width in loglog-scale in 1D setting with $H=0.7$, $\widehat{\Delta}=0.05$, and $\Delta t_{\mathrm{fine}}=\widehat{\Delta}/4$. Reference slope $n^{-\frac12}$ is shown in black. }
\label{fig:frac_error_vs_neurons}
\end{figure}
\begin{table}[htbp]
    \centering
    \renewcommand{\arraystretch}{1.4}
    \caption{Mean and standard deviations of validation losses and function errors for each hidden-layer width in 1D case. We have used $H=0.7$, $\widehat{\Delta}=0.05$, and $\Delta t_{\mathrm{fine}}=\widehat{\Delta}/4$.}
    \vspace{0.7em}
    {\footnotesize
    \begin{tabularx}{0.65\textwidth}{r *{3}{>{\centering\arraybackslash}X}}
        \hline
        \textbf{Width} &
        \textbf{$\norm{f}_{\alpha,\infty}$} &
        \textbf{$L^{2}(b)$} &
        \textbf{$L^{2}(\sigma)$} \\
        \hline
         8   & $0.5591 \pm 0.0888$ & $0.0428 \pm 0.0122$ & $0.0989 \pm 0.0277$ \\
        16   & $0.1279 \pm 0.0456$ & $0.0508 \pm 0.0106$ & $0.0172 \pm 0.0187$ \\
        32   & $0.0967 \pm 0.0204$ & $0.0388 \pm 0.0014$ & $0.0047 \pm 0.0018$ \\
        64   & $0.0726 \pm 0.0138$ & $0.0377 \pm 0.0004$ & $0.0026 \pm 0.0011$ \\
        128  & $0.0329 \pm 0.0106$ & $0.0185 \pm 0.0038$ & $0.0020 \pm 0.0010$ \\
        \hline
    \end{tabularx}
    }
    \label{tab:nets_hurst07}
\end{table}

To evaluate the fitting error arising from the estimation of Hurst index $H$, we consider several univariate settings with a fixed number of trajectories and varying numbers of observation points $M$. As $M$ increases, we expect the Hurst estimator to become more accurate, which should in turn reduce the error measured in the fractional-Sobolev norm. To isolate this effect, we do not re-estimate the drift and diffusion coefficients. Instead, after estimating the Hurst index from the simulated data, we use the true drift and diffusion functions. 

We repeat this procedure across multiple values of $M$ for the fixed parameters $H=0.7$, $\Delta t_{\mathrm{fine}}=\Delta t_{\mathrm{coarse}}/4$, and $N=2000$. Figure~\ref{fig:frac_error_vs_M} shows that the fitting error decreases as $M$ increases. We also include the corresponding theoretical upper bound for the fitting error as a function of $M$, as stated in Remark~\ref{remark:fitting-error}, showing that our numerical results are aligned with the theoretical upper bound.

\begin{figure}[h]
\centering
\hspace{-0.5cm} 
\includegraphics[width=0.77\textwidth,height=0.34\textheight]{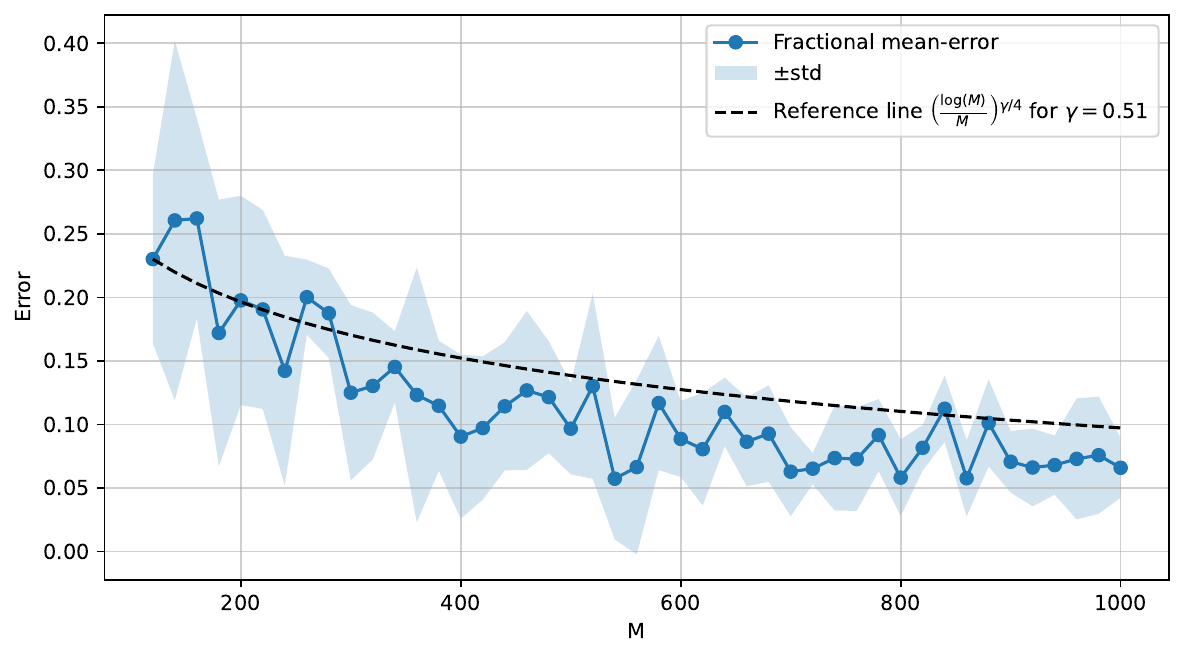}
\caption{Mean and standard deviations of fitting error as a function of $M$ in 1D setting with $H=0.7$, $\hat\Delta = 0.05$, $\Delta t_{\text{fine}} = \hat\Delta/4$. Theoretical upper bound $ \left(\frac{\log(M)}{M}\right)^{\gamma/4}$ for $\gamma = 0.51$ is shown in black.}
\label{fig:frac_error_vs_M}
\end{figure}

\pagebreak

\appendix

\section{Proof of Theorem \ref{thm:main}}\label{sec:apprendix}
\label{sec:proofs}
We split the proof into a series of lemmas and three propositions dealing with different sources of error. We begin with the time-discretisation error that follows directly from the existing results presented in the literature. Our main contributions consider approximation and fitting errors, dealt in Sections \ref{subsubsec:fitting} and \ref{subsubsec:approximation}.

\subsection{\bf Time-approximation error}
\label{subsubsec:time}
The time-discretisation error
\[\mathcal{E}_{\texttt{time},\alpha}(\Delta t) = \|{X}^{n,\hat{H}_M} - X^{n,\hat{H}_M,\Delta t}\|_{\alpha,\infty}\]
corresponding either to numerical errors arising from computations or from the fact that in practice one typically only have discrete but dense (high-frequency) observations. We note that the subject is well-studied in the literature, and the following result is a direct consequence of~\cite[Theorem 3.1]{Mishura-Shevchenko2008}, re-written in our context.
\begin{proposition}
\label{prop:time_error}
Let $(X_t^{n,M}; t\in[0,T])$ be the solution of Equation \eqref{eq:coefficient-approx}, and $(X_t^{n,\Delta t, M}; t \in[0,T])$ its Euler-Maruyama approximation with time-step $\Delta t$, and fix the estimator $\HeM\in[\underline{H},\overline{H}]$. Assume \ref{H5} holds true for $b_n,\sigma_n$ and $\HeM$.
Then, for any $\eta\in(0,1)$ and $\epsilon_0>0$ sufficiently small there exists a time-step $\Delta t_0$, an event $\Omega_{\eta,\epsilon_0}\subseteq\Omega$, and a random constant $C>0$ such that $\P(\Omega_{\eta,\epsilon_0})>1-\eta$ and 
\begin{align}\label{eq:rate_EM}
\mathcal{E}_{\texttt{time},\alpha}(\Delta t) &\le C\,  (\Delta t)^{2\overline{H}_M-1-\epsilon_0}, \text{ for all }\Delta t \le \Delta t_0,
\end{align} 
with $C$ independent on $\Delta t$.
\end{proposition}

Notice that, in addition to assumptions in \cite[Theorem 3.1]{Mishura-Shevchenko2008},  we only need to assume that the constants arising from the Lipschitz and H\"older continuity properties in \ref{H0}-\ref{H3} are uniformly bounded with respect to \( n \). This implies that convergence is independent of \( n \). More precisely, we can find a sufficiently small \( \Delta_0 \) that does not depend on \( n \), such that for any \( \Delta t \leq \Delta_0 \), convergence in \eqref{eq:rate_EM} still holds.

\subsection{\bf Fitting error}
\label{subsubsec:fitting}
For the sake of simplicity and without loss of generality, we only consider the case $d=1$. The general case follows by considering the below reasoning componentwise. 

We first quantify, in terms of the norm $\|\cdot\|_{1,1-\alpha}$, the distance between two fractional Brownian motions $B^{H_1}$ and $B^{H_2}$, corresponding to deterministic Hurst parameters $H_1$ and $H_2$. This preliminary estimate will play a key role in the analysis of the fitting error under parameter uncertainty.

\begin{lemma}\label{lem:brownian_convergence}
Let $H_1,H_2$ satisfying Assumption \ref{H0}. Then for any $\alpha\in(1-\underline{H}, 1/2)$ there exists a constant $C$ with all the moments finite such that
\begin{align}
&|B^{H_1} - B^{H_2}\|_{1,1-\alpha} \leq C|H_1-H_2|^{\frac12}.
\end{align}
\end{lemma}

\begin{proof}
We begin by computing $T_{s,t}(H_1,H_2) = \E[B_t^{H_1}-B_t^{H_2} - B_s^{H_1}-B_s^{H_2}]^2$. Using the stationarity of the increments, we get 
\begin{align*}
    T_{s,t}(H_1,H_2) &= 
    \E[(B_{t-s}^{H_1})^2] + \E[( B_{t-s}^{H_2})^2] - 2\E[(B_{t}^{H_1}-B_{s}^{H_1})(B_{t}^{H_2}-B_{s}^{H_2})]\\
    & = |t-s|^{2H_1} +|t-s|^{2H_2} \\
    &\quad - 2C_{H_1}C_{H_2}\E\left[\int_{-\infty}^{t-s}K_{H_1}(t-s,r)dW_r\int_{-\infty}^{t-s}K_{H_2}(t-s,r)dW_r\right].
\end{align*}

For the last term we have
\begin{align*}
&\E\left[\int_{-\infty}^{t-s}K_{H_1}(t-s,r)dW_r\int_{-\infty}^{t-s}K_{H_2}(t-s,r)dW_r\right] \nonumber\\
&\qquad = \int_{-\infty}^{t-s}\left\{(t-s-r)^{H_1-1/2} - (-r)_+^{H_1-1/2}\right\}\left\{(t-s-r)^{H_2-1/2} - (-r)_+^{H_2-1/2}\right\}dr\nonumber\\
&\qquad = (t-s)^{H_1+H_2}\int^{\infty}_{-1}\left\{(1+u)^{H_1-1/2} - u_+^{H_1-1/2}\right\}\left\{(1+u)^{H_2-1/2} - u_+^{H_2-1/2}\right\}du\nonumber\\
&\qquad = |t-s|^{H_1+H_2}\left\{\int^{\infty}_{0}\left\{(1+u)^{H_1-1/2} - u^{H_1-1/2}\right\}\right.\\
&\left.\hspace{0.5cm}\qquad \cdot \left\{(1+u)^{H_2-1/2} - u^{H_2-1/2}\right\}du+\frac{1}{H_1+H_2}\right\}\nonumber\\
&\qquad = -|t-s|^{H_1+H_2}\Gamma(-H_1-H_2)\left(\frac{\Gamma(1/2+H_1)}{\Gamma(1/2-H_2)} + \frac{\Gamma(1/2+H_2)}{\Gamma(1/2-H_1)}\right). 
\end{align*}
Therefore,
\begin{align*}
    T_{s,t}(H_1,H_2)
    & = |t-s|^{2(H_1\wedge H_2)}\left(1 +|t-s|^{2|H_2-H_1|} - 2|t-s|^{|H_2-H_1|}f(H_1,H_2)\right)\\
    &= |t-s|^{2H_1} + |t-s|^{2H_2} - 2|t-s|^{H_2+H_1}f(H_1,H_2)
\end{align*}
with
$$
f(H_1,H_2):=-{\Gamma(-H_1-H_2)}C_{H_1}C_{H_2}\left(\frac{\Gamma(1/2+H_1)}{\Gamma(1/2-H_2)} + \frac{\Gamma(1/2+H_2)}{\Gamma(1/2-H_1)}\right).
$$
Notice that, for $H_1,H_2\in(\underline{H},\overline{H})$, the map $(H_1,H_2)\mapsto f(H_1,H_2)$ is continuous, positive, bounded by $1$, and converges to $1$ when $H_2\rightarrow H_1$. 
In order to obtain explicit rate, we rely on Taylor's expansion of the terms. Denote $v=|t-s|$ and, without loss of generality, let $H_2<H_1$.
Then we have 
\begin{align*}
&T_{s,t}(H_1,H_2)\\
&= v^{2H_2} +v^{2H_1} - 2v^{H_2+H_1}f(H_1,H_2)\\
&= v^{2H_2}(1 +v^{2(H_1-H_2)} - 2v^{H_1-H_2}) +2v^{H_2+H_1}(1- f(H_1,H_2))\\
&= v^{2H_2}(1 -v^{H_1-H_2})^2  \\
&\quad +2v^{H_2+H_1}\left[1+C_{H_1}C_{H_2}\left(B(1/2+H_1, -H_1-H_2)+ B(1/2+H_2,-H_1-H_2)\right)\right].
\end{align*}
By using the inequality $xe^{1-x}\leq 1$ for $x\in[0,1]$, we observe that, for $v\in(0,1)$,
$$
1-v^{H_1-H_2}\leq (H_1-H_2)|\log(v)|.
$$ 
This leads to
\begin{align*}
   v^{2H_2}(1 -v^{H_1-H_2})^2 \leq v^{2H_2}(H_1-H_2)^2|\log(v)|^2 \lesssim v^{2H_2-\epsilon}(H_1-H_2)^2
\end{align*}
for any $\epsilon>0$. For the other term, by expanding \\
$H_1\mapsto C_{H_1}\left(B(1/2+H_1, -H_1-H_2)+ B(1/2+H_2,-H_1-H_2)\right)$ around $H_2$ gives
\begin{align*}
&C_{H_1}\left(B(1/2+H_1, -H_1-H_2)+ B(1/2+H_2,-H_1-H_2)\right) \\
&\quad =  -\frac1{C_{H_2}}+ (H_1-H_2)\left[-\frac1{C_{H_2}^2}\dfrac{d}{dx}C_{x}|_{x=H_2}+ C_{H_2}\dfrac{d}{dx}B(1/2+x, -x-H_2)|_{x=H_2}\right.\\
&\qquad \left.+ C_{H_2}\dfrac{d}{dx}B(1/2+H_2,-x-H_2)|_{x=H_2}\right] + O((H_1-H_2)^2),
\end{align*}
where the derivatives are bounded since $H_2\in (\underline{H},\overline{H})\subset(1/2,1)$. 
It follows by combining both expansions that, for any $\epsilon>0$, there exists a constant $C>0$, depending on $[0,T]$, $H_1$, $H_2$, and $\epsilon$, such that
\begin{align*}
T_{s,t}(H_1,H_2)&\leq C |t-s|^{2H_2-\epsilon} (H_1-H_2).
\end{align*}
In order to complete the proof, note that we have 
$$
\Vert B^{H_1}-B^{H_2}\Vert_{1,1-\alpha} \lesssim \Vert B^{H_1}-B^{H_2}\Vert_{1-\alpha+\epsilon},
$$
where $\Vert \cdot\Vert_{\gamma}$ denotes the $\gamma$-H\"older norm. By considering a process 
$$
X_t = \frac{B_t^{H_1}-B_t^{H_2}}{\sqrt{H_1-H_2}},
$$
it then follows from \cite[Corollary 2.11]{nummi2024necessary} that 
$$
|X_t-X_s| \leq C(\omega) |t-s|^{\gamma}
$$
for any $\gamma<H_2$, and where $C(\omega)$ has all moments finite. Consequently, we obtain
$$
\Vert B^{H_1}-B^{H_2}\Vert_{1-\alpha+\epsilon} \lesssim C(\omega)\sqrt{H_1-H_2} 
$$
provided that $\epsilon>0$ is small enough and $\gamma$ large enough so that $1-\alpha+\epsilon<\gamma$. This concludes the proof. 
\end{proof}
The bound for the fitting error
 $\mathcal{E}_{\texttt{fit},\alpha}(M) =\Vert X^{\HeM} - {X}^H\Vert_{\alpha,\infty}$
 now follows almost instantly from Lemma \ref{lem:brownian_convergence}.

\begin{proposition}\label{prop:fitting_error}
Consider $(X_t^{\HeM};t\in[0,T])$ and $(X_t^{{H}};t\in[0,T])$ the statistical model in \eqref{eq:statistical_mode} and the true model in \eqref{eq:true_model}, respectively. Under assumptions of Section \ref{sec:assumptions} there exist $1/2<\underline{H}<\overline{H}<1$ such that $\underline{H}<H<\overline{H}$. Let $\HeM$ the estimator of $H$. Then, for any $\alpha\in(0,1/2)$,
there exists a random variable $C$, independent of $M$, such that
\begin{equation*}
\mathcal{E}_{\texttt{fit},\alpha}(M) \leq C|\HeM-H|^{\frac12}.
\end{equation*}
\end{proposition}
\begin{proof}
Using that $\HeM\in[\underline{H},\overline{H}]$ and \cite[Theorem 2]{Viitasaari-2022}, we obtain that there exists a random constant $C_\omega$ such that
\begin{align*}
\mathcal{E}_{\texttt{fit},\alpha}(M)\leq C_\omega\|B^{\HeM} - B^{H}\|_{1,1-\alpha}.
\end{align*}
The claim now follows from Lemma \ref{lem:brownian_convergence}. 
\end{proof}

\subsection{\bf Coefficients approximation error}
\label{subsubsec:approximation}
To simplify the notation in this section, the processes $X^{\HeM}$ and $X^{n,\HeM}$ will be denoted by $X$ and $X^n$, respectively, and the fractional Brownian motions will be denoted simply by $B$. However, it should be kept in mind that the underlying Hurst parameter is $\HeM$. In order to simplify the notation, in the sequel we denote $f \lesssim g$ if $f\leq Cg$ for some (possibly random) unimportant constant $C$.
\begin{proposition}
\label{prop:coefficient-approximation}
Suppose that the Assumptions of Section \ref{sec:assumptions} are valid. Let $\tilde{\delta} < \HeM$, $s,s_2\in (0,1)$ be such that we can choose $\alpha \in (1-\HeM, \min(1/2,s\tilde{\delta},s_2))$. Set $\tilde\sigma = \sigma - \sigma_n$ and $\tilde b = b - b_n$. Then for any $\lambda \geq 0$ and any 
$q>\frac{d}{1-s}$, $q_2>\frac{1}{1-s_2}$, and $\rho > \frac{1}{1-\alpha}$ such that $q\ge\rho$, we have, for suitably large $\lambda$, that
 \begin{align*}
   \Vert X_\cdot - X_\cdot^n\Vert_{\alpha,\lambda} 
    &\lesssim \Vert\partial_t\nabla_z\tilde\sigma\Vert_{(q_2,q)} +\Vert\nabla_z\tilde\sigma\Vert_{(q_2,q)} + \Vert\partial_t\tilde\sigma\Vert_{(q_2,q)} + \Vert\tilde\sigma\Vert_{(q_2,q)}  \\
    &+\Vert \nabla_z \tilde b\Vert_{(\rho,q)}  + \Vert  \tilde b\Vert_{(\rho,q)} . 
\end{align*}
    Consequently, we have
     \begin{align*}
   \Vert X_\cdot - X_\cdot^n\Vert_{\alpha,\infty} 
    &\lesssim \Vert\partial_t\nabla_z\tilde\sigma\Vert_{(q_2,q)} +\Vert\nabla_z\tilde\sigma\Vert_{(q_2,q)} + \Vert\partial_t\tilde\sigma\Vert_{(q_2,q)} + \Vert\tilde\sigma\Vert_{(q_2,q)}  \\
    &+\Vert \nabla_z \tilde b\Vert_{(\rho,q)}  + \Vert  \tilde b\Vert_{(\rho,q)} . 
\end{align*}
\end{proposition}

We use short notation 
$$
G_t^{\sigma}(f) = \int_0^t \sigma(s,f(s))dB_s
$$
and 
$$
F_t^{b}(f)  = \int_0^t b(s,f(s))ds.
$$
Then we can write, with $\sigma_n$ and $b_n$ approximating $\sigma$ and $b$,
\begin{equation}
\label{eq:solution}
X_t = x_0 + \int_0^t b(s,X_s)ds +\int_0^t \sigma(s,X_s)dB_s =x_0 + F_t^b(X)+G_t^{\sigma}(X)
\end{equation}
and
\begin{equation}
\label{eq:approx-solution}
X_t^n = x_0 + \int_0^t b_n(s,X^n_s)ds +\int_0^t \sigma_n(s,X^n_s)dB_s =x_0 + F_t^{b_n}(X^n)+G_t^{\sigma_n}(X^n).
\end{equation}
We begin with some auxiliary estimates.

\begin{lemma}
\label{lemma:difference}
    Let $s\in (0,1)$ be arbitrary and $q>\frac{d}{1-s}$. Let $f:\mathbb{R}^d \mapsto \mathbb{R}$ satisfy $f \in W^{1,q}_0(\mathbb{R}^d)$. Then for all Lebesgue points $y,z$ of $f$ we have 
    $$
    |f(y)-f(z)| \lesssim |y-z|^s \Vert \nabla f\Vert_q.
    $$
\end{lemma}
\begin{proof}
From fractional maximal function inequality, see e.g. \cite[Lemma C.2]{HTV}, we have, for any $s\in(0,1)$,
    \begin{align*}
    |f(y)-f(z)| 
    &\lesssim |y-z|^s \left(I_{1-s}|\nabla f|(y) + I_{1-s}|\nabla f|(z)\right)\\
    &\leq 2|y-z|^s \sup_{x\in \text{supp}(f)}I_{1-s}|\nabla f|(x),
    \end{align*}
     where 
    $$
    I_{1-s}|\nabla f|(x) = \int_{\mathbb{R}^d} |x-y|^{-d+1-s}|\nabla f(y)|dy
    $$
    denotes the Riesz potential. Choose $1<p<\frac{d}{d-1+s}$ and let $q$ be the H\"older conjugate of $p$. Then H\"older inequality gives
    \begin{align*}
    &\int_{\mathbb{R}^d} |x-y|^{-d+1-s}|\nabla f(y)|dy \\
    & = \int_{\text{supp}(f)} |x-y|^{-d+1-s}|\nabla f(y)|dy \\
    &\leq \left(\int_{\text{supp}(f)}|x-y|^{(-d+1-s)p}dy\right)^{\frac{1}{p}}\Vert \nabla f\Vert_q.
    \end{align*}
Here 
$$
\sup_{x\in \mathbb{R}^d}\int_{\text{supp}(f)}|x-y|^{(-d+1-s)p}dy \lesssim |\text{supp}(f)|+1
$$
for any   $1<p<\frac{d}{d-1+s}$ which translates into $q>\frac{d}{1-s}$. This completes the proof.
\end{proof}
\begin{lemma}
\label{lemma:simpler}
Let $f\in L^{q}(\mathcal{K})$ for an arbitrary $\mathcal{K}\subset \mathbb{R}^d$. Then for any $q\geq 1$ there exists $x_0$ such that 
$$
|f(x_0)| \lesssim \Vert f\Vert_q.
$$
\end{lemma}
\begin{proof}
    If for almost all $x_0$ we would have $f(x_0)> 2 \Vert f\Vert_q$, it would follow that $\Vert f\Vert_q >2\Vert f\Vert_q$ giving the contradiction.
\end{proof}
The following follows in a similar manner.
\begin{lemma}
\label{lemma:simpler-v2}
For any $q_2,q\geq 1$ there exists $s_0\in [0,T]$ and $z_0\in\mathcal{K}$ such that
$$
|\sigma(s_0,z_0)| \lesssim \Vert\sigma\Vert_{(q_2,q)}.
$$
\end{lemma}
\begin{lemma}
\label{lemma:difference-sigma}
    Let $s,s_2\in (0,1)$ be arbitrary. Let $q>\frac{d}{1-s}$ and $q_2 > \frac{1}{1-s_2}$. Suppose $\sigma$ is supported on a compact set $[0,T] \times \mathcal{K}$ and that partial weak derivatives exists.     
    Then we have 
   \begin{align*}
       |\sigma(r,x)-\sigma(u,z)| &\lesssim |x-z|^{s} \left[\Vert\partial_t\nabla_z\sigma\Vert_{(q_2,q)} + \Vert\nabla_z\sigma\Vert_{(q_2,q)}\right] \\
       &+|r-u|^{s_2}\left[ \Vert\partial_t\nabla_z\sigma\Vert_{(q_2,q)}+ \Vert\partial_t\sigma\Vert_{(q_2,q)}\right].
   \end{align*}

\end{lemma}
\begin{proof}
By triangle inequality we get
$$
|\sigma(r,x)-\sigma(u,z)| \leq |\sigma(r,x)-\sigma(r,z)| + |\sigma(r,z)-\sigma(u,z)|.
$$
where, by Lemma \ref{lemma:difference},
$$
|\sigma(r,x)-\sigma(r,z)| \lesssim |x-z|^{s}\Vert \nabla_z\sigma(r,\cdot)\Vert_q
$$
with $s\in (0,1)$ and $q>\frac{d}{1-s}$.
Since, by Lemma \ref{lemma:difference} again,
\begin{align*}
|\nabla_z \sigma(r,z) - \nabla_z\sigma(v,z)| &\lesssim |r-v|^{s_2}\Vert \partial_t\nabla_z \sigma(\cdot,z)\Vert_{q_2}    
\end{align*}
for any $s_2\in(0,1)$ and $q_2 > 1/(1-s_2)$, we get 
\begin{align*}
    \Vert \nabla_z\sigma(r,\cdot)\Vert_q & \leq \Vert \nabla_z\sigma(r,\cdot)-\nabla_z\sigma(r_0,\cdot) \Vert_q + \Vert \nabla_z\sigma(r_0,\cdot)\Vert_q\\
    &\lesssim \Vert\partial_t\nabla_z\sigma\Vert_{(q_2,q)} + \Vert \nabla_z\sigma(r_0,\cdot)\Vert_q
\end{align*}
where we have also used Minkowski's integral equality. By Lemma \ref{lemma:simpler} we can choose $r_0$ such that 
$$
\Vert \nabla_z\sigma(r_0,\cdot)\Vert_q \leq  \Vert\nabla_z\sigma\Vert_{(q_2,q)}
$$
which then gives us 
$$
|\sigma(r,x)-\sigma(r,z)| \lesssim |x-z|^{s}\left[ \Vert\partial_t\nabla_z\sigma\Vert_{(q_2,q)}+ \Vert\nabla_z\sigma\Vert_{(q_2,q)}\right].
$$
Similarly for the other term, \cref{lemma:difference} gives
$$
|\sigma(r,z)-\sigma(u,z)| \lesssim |r-u|^{s_2}\Vert \partial_t\sigma(\cdot,z)\Vert_{q_2}
$$
where now 
\begin{align*}
    \Vert \partial_t\sigma(\cdot,z)\Vert_{q_2} & \leq \Vert \partial_t\sigma(\cdot,z)- \partial_t\sigma(\cdot,z_0)\Vert_{q_2} + \Vert \partial_t\sigma(\cdot,z_0)\Vert_{q_2}\\
    &\lesssim  \Vert\partial_t\nabla_z\sigma\Vert_{(q_2,q)} +\Vert\partial_t\sigma\Vert_{(q_2,q)}
\end{align*}
with suitable choice of $z_0$.
This gives 
$$
|\sigma(r,z)-\sigma(u,z)| \lesssim |r-u|^{s_2}\left[ \Vert\partial_t\nabla_z\sigma\Vert_{(q_2,q)}+\Vert\partial_t\sigma\Vert_{(q_2,q)}\right]
$$
from which the result follows.
\end{proof}

Notice that from Theorem 2.1 in \cite{Nualart-Rascanu2002}, under assumptions \ref{H0}-\ref{H2}, the solutions to Equations \eqref{eq:solution}-\eqref{eq:approx-solution} exist and are unique. Next two propositions bound terms $G^\sigma(X)$ and $F^b(X)$ in the norm $\Vert \cdot \Vert_{\alpha,\lambda}$ in terms of $\Vert \cdot\Vert_{(p,q)}$ norms of $\sigma$ and $b$ and their partial derivatives.
\begin{proposition}
\label{prop:G-bound}
Let $\tilde{\delta} < \HeM$, $s,s_2\in (0,1)$ be such that we can choose $\alpha \in (1-\HeM, \min(1/2,s\tilde{\delta},s_2))$. Then for any $\lambda \geq 0$ and any 
$q>\frac{d}{1-s}$, $q_2>\frac{1}{1-s_2}$
we have
$$
\Vert G^\sigma(X)\Vert_{\alpha,\lambda} \lesssim  \Vert\partial_t\nabla_z\sigma\Vert_{(q_2,q)}+\Vert\nabla_z\sigma\Vert_{(q_2,q)}+\Vert\partial_t\sigma\Vert_{(q_2,q)} +  \Vert\sigma\Vert_{(q_2,q)}.
$$
\end{proposition}
We remark that in the above proposition the constant depends also on the support of $X$. In particular, the constant is random. 
\begin{proof}[Proof of Proposition \ref{prop:G-bound}]
Under \ref{H1} we apply Rademacher Theorem that guarantees that partial derivatives $\partial_t\sigma$ and $\partial_t\nabla \sigma$ exist almost everywhere. On the other hand, since $t\in [0,T]$, we can choose a compact set $\mathcal{K}$ containing $X$ and by restriction, $\sigma$ can be assumed to be supported on $[0,T] \times \mathcal{K}$. Recall also that the solution $X$ is H\"older continuous of any order $\tilde{\delta}<H$.
Now by \cite[Eq. (4.12)]{Nualart-Rascanu2002} we have 
\begin{align*}
&|G^\sigma_t(X)| + \int_0^t \frac{|G^\sigma_t(X)-G^\sigma_s(X)|}{(t-s)^{\alpha+1}}ds \\
&\lesssim \Lambda_\alpha(B)\int_0^t \left((t-r)^{-2\alpha}+r^{-\alpha}\right)\left(|\sigma(r,X_r)|+\int_0^r \frac{|\sigma(r,X_r)-\sigma(y,X_y)|}{(r-y)^{\alpha+1}}dy\right)dr.
\end{align*}
Set 
$$
\Delta(\sigma) =  \Vert\partial_t\nabla_z\sigma\Vert_{(q_2,q)}+ \Vert\nabla_z\sigma\Vert_{(q_2,q)}+ \Vert\partial_t\sigma\Vert_{(q_2,q)}.
$$
Using also H\"older continuity of $X$, Lemma \ref{lemma:difference-sigma} implies 
$$
|\sigma(r,X_r)-\sigma(y,X_y)| \lesssim \Delta(\sigma)\left[|X_r-X_y|^s + |r-y|^{s_2}\right] \leq [X]_{\tilde\delta} \Delta(\sigma)|r-y|^{\min(s\tilde\delta,s_2)}.
$$
Hence we obtain, for any $\alpha < \min(1/2,s\tilde\delta,s_2)$ with $\tilde\delta <\HeM$, that 
\begin{align*}
&\int_0^t \left((t-r)^{-2\alpha}+r^{-\alpha}\right)\int_0^r \frac{|\sigma(r,X_r)-\sigma(y,X_y)|}{(r-y)^{\alpha+1}}dy dr\\
&\lesssim  \Delta(\sigma) \int_0^t \left((t-r)^{-2\alpha}+r^{-\alpha}\right)dydr\\
&\lesssim  \Delta(\sigma).
\end{align*}
For the other term we use the pinning argument. That is, with arbitrary $t_0$ we get from Lemma \ref{lemma:difference-sigma}
\begin{align*}
|\sigma(r,X_r)| &\leq |\sigma(r,X_r)-\sigma(t_0,z_0)| + |\sigma(t_0,z_0)|\\
&\lesssim \Delta(\sigma) + |\sigma(t_0,z_0)| \\
&\lesssim \Delta(\sigma) + \Vert\sigma\Vert_{(q_2,q)}
\end{align*}
where we have chosen $t_0$ and $z_0$ as in \cref{lemma:simpler-v2}. This completes the proof.
\end{proof}
\begin{proposition}
\label{prop:F-bound}
Let $\tilde\delta < \HeM$, $s\in (0,1)$ be such that we can choose $\alpha \in (1-\HeM, \min(1/2,s\tilde\delta))$. Then for any $\lambda \geq 0$ and any 
$q>\frac{d}{1-s}$, $\rho>\frac{1}{1-\alpha}$, with $q\ge \rho$,
we have
    $$
    \Vert F_\cdot^b(X)\Vert_{\alpha,\lambda} \lesssim \Vert \nabla_z b\Vert_{(\rho,q)} + \Vert  b\Vert_{(\rho,q)}. 
    $$
\end{proposition}
\begin{proof}
By \cite[Eq. (4.22)]{Nualart-Rascanu2002} we have 
\begin{align*}
&|F^b_t(X)| + \int_0^t \frac{|F^b_t(X)-F^b_s(X)|}{(t-s)^{\alpha+1}}ds \\
&\lesssim \int_0^t |b(u,X_u)|(t-u)^{-\alpha}du.
\end{align*}
Since $b$ is locally Lipschitz, it is differentiable almost everywhere. Then 
\cref{lemma:difference} gives, for any $q>\frac{d}{1-s}$,
$$
|b(u,X_u) - b(u,z)| \leq |X_u-z|^s\Vert \nabla_z b(u,\cdot)\Vert_q \lesssim \Vert\nabla_z b(u,\cdot)\Vert_q,
$$
where the constant depends on $\mathcal{K}$.
From this together with pinning argument and H\"older inequality we obtain
\begin{align*}
&|F^b_t(X)| + \int_0^t \frac{|F^b_t(X)-F^b_s(X)|}{(t-s)^{\alpha+1}}ds \\
&\lesssim \int_0^t |b(u,X_u)|(t-u)^{-\alpha}du \\
&\lesssim \int_0^t \Vert\nabla_z b(u,\cdot)\Vert_q (t-u)^{-\alpha}du + \int_0^t |b(u,z)|(t-u)^{-\alpha}du \\
&\lesssim \Vert \nabla_z b(u,\cdot)\Vert_{(\rho,q)} + \Vert b(\cdot,z)\Vert_{\rho},
\end{align*}
since now 
$$
\int_0^t (t-u)^{-\alpha \frac{\rho}{\rho-1}}du < \infty,\text{ for any }\rho> \frac{1}{1-\alpha}.
$$
Now, by \cref{lemma:simpler} we can choose $z$ such that 
$$
\Vert b(\cdot,z)\Vert_\rho \lesssim \Vert\Vert b\Vert_{\rho(t)} \Vert_q \leq \Vert b\Vert_{(\rho,q)}
$$
where the second inequality follows from Minkowski's integral inequality that can be applied since $\rho \leq q$. This completes the proof. 
\end{proof}
\begin{lemma}
\label{lemma:uniform-boundedness}
Suppose that the Assumptions of Section \ref{sec:assumptions} are valid and let $\alpha$ be as in Proposition \ref{prop:coefficient-approximation}. Then  
$$
\sup_n \Vert X^n\Vert_{\alpha,\infty}<\infty.
$$
In particular, the $\tilde\delta$-H\"older norm is uniformly bounded for any $\tilde\delta > \alpha$ and consequently, there exists $N$ such that $X^n \in [-N,N]^d$ for all $n\in \mathbb{N}$.
\end{lemma}
\begin{proof}
By \cite[Proposition 5.1]{Nualart-Rascanu2002} we have 
$$
\Vert X^n\Vert_{\alpha,\infty} \leq C_{1,n} \exp(C_{2,n}\Lambda_{\alpha}(B))
$$
where now $C_{1,n}$ and $C_{2,n}$ depend only on $\alpha$, $T$, and the Lipschitz and H\"older constants appearing in Assumptions \ref{H1}-\ref{H3}. As by assumption these are uniformly bounded, it follows that $C_{1,n}$ and $C_{2,n}$ can be chosen independently of $n$. Finally, the embedding of H\"older space into $W_{\alpha,\infty}$ gives that for the H\"older seminorm $[\cdot]_{\tilde\delta}$ we also have
$$
\sup_{n}[X^n]_{\tilde\delta} < \infty
$$
for any $\tilde\delta>\alpha$. The final claim now follows trivially from this.
\end{proof}
We are now ready to prove Proposition \ref{prop:coefficient-approximation}.
\begin{proof}[Proof of Proposition \ref{prop:coefficient-approximation}]
Without loss of generality and for the sake of simplicity, we only consider the case $d=1$. The general case can be then treated with the same arguments by consider each components separately. From \eqref{eq:solution}-\eqref{eq:approx-solution} we have
\begin{align*}
    X_t - X_t^n &= F_t^b(X)-F_t^{b_n}(X^n) + G_t^\sigma(X)-G_t^{\sigma_n}(X^n)\\
    &=F_t^b(X)-F_t^{b_n}(X) + F_t^{b_n}(X)-F_t^{b_n}(X^n)\\
    &+ G_t^\sigma(X)-G_t^{\sigma_n}(X) + G_t^{\sigma_n}(X)-G_t^{\sigma_n}(X^n) \\
    &=F_t^{b-b_n}(X) + F_t^{b_n}(X)-F_t^{b_n}(X^n)\\
    &+ G_t^{\sigma-\sigma_n}(X) + G_t^{\sigma_n}(X)-G_t^{\sigma_n}(X^n),
\end{align*}
leading to
\begin{align*}
    \Vert X_\cdot - X_\cdot^n\Vert_{\alpha,\lambda} 
    &\leq  \Vert F_\cdot^{b-b_n}(X)\Vert_{\alpha,\lambda} + \Vert G_\cdot^{\sigma-\sigma_n}(X)\Vert_{\alpha,\lambda} \\
    &+\Vert F_\cdot^{b_n}(X)-F_\cdot^{b_n}(X^n)\Vert_{\alpha,\lambda} + \Vert G_\cdot^{\sigma_n}(X)-G_\cdot^{\sigma_n}(X^n)\Vert_{\alpha,\lambda}.
\end{align*}
By using \cite[Proposition 4.4]{Nualart-Rascanu2002} we get
$$
\Vert F_\cdot^{b_n}(X)-F_\cdot^{b_n}(X^n)\Vert_{\alpha,\lambda} \leq \frac{d_{N,n}}{\lambda^{1-\alpha}}\Vert X_\cdot - X_\cdot^n\Vert_{\alpha,\lambda},
$$
where $N$ is chosen so large that both $X$ and $X^n$ are contained in $[-N,N]$, and $d_{N,n}$ depends on $\alpha$, $T$, and the Lipschitz constant of $b_n$ (restricted on $[-N,N]$). By Lemma \ref{lemma:uniform-boundedness} we can now take $N$ to be independent of $n$ and since Lipschitz constants are uniformly bounded, we have
$$
\Vert F_\cdot^{b_n}(X)-F_\cdot^{b_n}(X^n)\Vert_{\alpha,\lambda} \leq \frac{d_{N}}{\lambda^{1-\alpha}}\Vert X_\cdot - X_\cdot^n\Vert_{\alpha,\lambda}.
$$
Similarly, for the fourth term \cite[Proposition 4.2]{Nualart-Rascanu2002} gives that  
$$
\Vert G_\cdot^{\sigma_n}(X)-G_\cdot^{\sigma_n}(X^n)\Vert_{\alpha,\lambda} \leq \frac{C(B)C_{N,n}}{\lambda^{1-2\alpha}}(1+\Delta(X)+\Delta(X^n))\Vert X_\cdot - X_\cdot^n\Vert_{\alpha,\lambda},
$$
where $C(B)$ depends only on $B$, $N$ is chosen so large that both $X$ and $X^n$ are contained in $[-N,N]$, $C_{N,n}$ depends on $\alpha$, $T$, and the Lipschitz constants of $\sigma_n$ and $\nabla \sigma_n$ (restricted on $[-N,N]$), and  
$$
\Delta(f) = \sup_t \int_0^t \frac{|f(t)-f(s)|}{|t-s|^{1+\alpha}}ds.
$$
Using 
$$
\Delta(f) \leq [f]_{\tilde\delta} \sup_{t\in[0,T]}\int_0^t (t-s)^{ \tilde\delta-1-\alpha}ds \lesssim [f]_{\tilde\delta}
$$
for any $\tilde\delta \in (\alpha,\HeM)$ and Lemma \ref{lemma:uniform-boundedness}, the above reasoning gives us 
$$
\Vert G_\cdot^{\sigma_n}(X)-G_\cdot^{\sigma_n}(X^n)\Vert_{\alpha,\lambda} \leq \frac{C(B)C_{N}}{\lambda^{1-2\alpha}}\Vert X_\cdot - X_\cdot^n\Vert_{\alpha,\lambda}.
$$
This implies that
\begin{align*}
    \Vert X_\cdot - X_\cdot^n\Vert_{\alpha,\lambda} 
    &\leq \Vert F_\cdot^{b-b_n}(X)\Vert_{\alpha,\lambda} + \Vert G_\cdot^{\sigma-\sigma_n}(X)\Vert_{\alpha,\lambda} \\
    &+ \left[\frac{d_{N}}{\lambda^{1-\alpha}} + \frac{C(B)C_{N}}{\lambda^{1-2\alpha}}\right] \Vert X_\cdot - X_\cdot^n\Vert_{\alpha,\lambda},
\end{align*}
which, together with Proposition \ref{prop:G-bound} and Proposition \ref{prop:F-bound}, leads to 
\begin{align*}
   &\left[1-\frac{d_{N}}{\lambda^{1-\alpha}} - \frac{C(B)C_{N}}{\lambda^{1-2\alpha}}\right] \Vert X_\cdot - X_\cdot^n\Vert_{\alpha,\lambda} \\
    &\lesssim \Vert \partial_t\nabla_z\tilde\sigma\Vert_{(q_2,q)}+\Vert\nabla_z\tilde\sigma\Vert_{(q_2,q)}+\Vert \partial_t\tilde\sigma\Vert_{(q_2,q)} + \Vert \Vert\tilde\sigma\Vert_{(q_2,q)} \\
    &+\Vert \nabla_z \tilde b \Vert_{(\rho,q)} + \Vert  \tilde b\Vert_{(\rho,q)}. 
\end{align*}
Choosing now $\lambda$ large enough yields the first claim, from which the second follows by equivalence of norms $\Vert \cdot\Vert_{\alpha,\lambda}$ and $\Vert \cdot\Vert_{\alpha,\infty}$. This completes the whole proof.
\end{proof}

\end{document}